\documentclass[11pt,a4paper,reqno]{amsart}
\usepackage{amsfonts,amssymb,amsmath,amsthm,amscd,latexsym,graphicx,url, color,xspace,verbatim}
\usepackage[english]{babel}
\usepackage{a4wide}

\usepackage[margin=1in]{geometry}

\usepackage{hyperref}
\newtheorem{theorem}{Theorem}[section]
\newtheorem{corollary}[theorem]{Corollary}
\newtheorem{lemma}[theorem]{Lemma}
\newtheorem{proposition}[theorem]{Proposition}

\theoremstyle{definition}
\newtheorem{definition}[theorem]{Definition}

\newcommand{\be}{\begin{equation}}
\newcommand{\bel}[1]{\begin{equation}\label{#1}}
\newcommand{\ee}{\end{equation}}
\newcommand{\barr}{\begin{eqnarray}}
\newcommand{\earr}{\end{eqnarray}}
\newcommand{\bars}{\begin{eqnarray*}}
\newcommand{\ears}{\end{eqnarray*}}

\newtheorem{subn}{\name}

\newcommand{\bsn}[1]{\def\name{#1}\begin{subn}}
\newcommand{\esn}{\end{subn}}
\newtheorem{sub}{\name}[section]

\newcommand{\bs}{\begin{sub}}
\newcommand{\es}{\end{sub}}
\newcommand{\bsl}[1]{\begin{sub}\label{#1}}

\newcommand{\bth}[1]{\def\name{Theorem}
\begin{sub}\label{t:#1}}
\newcommand{\blemma}[1]{\def\name{Lemma}
\begin{sub}\label{l:#1}}
\newcommand{\bcor}[1]{\def\name{Corollary}
\begin{sub}\label{c:#1}}
\newcommand{\bdef}[1]{\def\name{Definition}
\begin{sub}\label{d:#1}}
\newcommand{\bprop}[1]{\def\name{Proposition}
\begin{sub}\label{p:#1}}

\newcommand{\BA}{\begin{array}}
\newcommand{\EA}{\end{array}}
\newcommand{\BAN}{\renewcommand{\arraystretch}{1.2}
\setlength{\arraycolsep}{2pt}\begin{array}}
\newcommand{\BAV}[2]{\renewcommand{\arraystretch}{#1}
\setlength{\arraycolsep}{#2}\begin{array}}
\newcommand{\BSA}{\begin{subarray}}
\newcommand{\ESA}{\end{subarray}}

\newcommand{\BAL}{\begin{aligned}}
\newcommand{\EAL}{\end{aligned}}
\newcommand{\BALG}{\begin{alignat}}
\newcommand{\EALG}{\end{alignat}}
\newcommand{\BALGN}{\begin{alignat*}}
\newcommand{\EALGN}{\end{alignat*}}
\newcommand{\note}[1]{\noindent\textit{#1.}\hspace{2mm}}
\newcommand{\Proof}{\note{Proof}}
\newcommand{\forevery}{\quad \forall}

\newcommand{\norm}[1]{\left \|#1\right \|}

\def\angb<#1>{\langle #1 \rangle}%% angle bracket
\newcommand{\supp}{\opname{supp}}

\newcommand{\prt}{\partial}
\newcommand{\sms}{\setminus}

\newcommand{\tl}{\tilde}
\newcommand{\sbs}{\subset}

\newcommand{\ovl}{\overline}
\newcommand{\unl}{\underline}

\def\ga{\alpha}     \def\gb{\beta}       \def\gg{\gamma}
       \def\gd{\delta}      \def\ge{\epsilon}
\def\gth{\theta}                         \def\vge{\varepsilon}
\def\gf{\phi}       \def\vgf{\varphi}    
      \def\gk{\kappa}      \def\gl{\lambda}

      \def\gw{\omega}
                
     \def\Gd{\Delta}

\def\Gw{\Omega}              

   \def\CO{{\mathcal O}}   
\def\CA{{\mathcal A}}      
   \def\CE{{\mathcal E}}   \def\CF{{\mathcal F}}
      
   \def\CK{{\mathcal K}}   \def\CL{{\mathcal L}}

\def\BBD {\mathbb D}

   \def\BBN {\mathbb N}    
   \def\BBR {\mathbb R}

\def\R{\mathbb{R}}

\let\ol=\overline

\let\.=\cdot
\let\0=\emptyset

\def\O{\Omega}

\def\loc{\text{\rm loc}}

\def\supp{\text{\rm supp}}

\newcommand{\su}[2]{\genfrac{}{}{0pt}{}{#1}{#2}}

\numberwithin{equation}{section}

\title[On the generalized principal eigenvalue of Quasilinear Operator]{On the generalized principal eigenvalue of quasilinear operators}

\author{Phuoc-Tai Nguyen}
\address{Departamento de Matem\'atica\\ 
Pontificia Universidad Cat\'olica de Chile, Avda. Vicu\~na Mackenna 4860, Santiago, Chile }
 \email{nguyenphuoctai.hcmup@gmail.com}
\author{Hoang-Hung Vo}
\address{Faculty of Mathematics and Computer Science, University of Science, Ho Chi Minh City National University, No. 227 Nguyen Van Cu Street, Ward 4, District 5, Ho Chi Minh City, Vietnam}
\email{vhhung@hcmus.edu.vn}

\begin{document}

\maketitle

\begin{abstract}
 The notions of generalized principal eigenvalue for linear second order elliptic operators in general domains introduced by Berestycki et al. \cite{BNV,BR0,BR3}  have become a very useful and important tool in analysis of partial differential equations. In this paper, we extend these notions for quasilinear operator of the form
$$\CK_V[u]:=-\Delta_p u +Vu^{p-1},\quad\quad u \geq0.$$
This operator is a natural generalization of self-adjoint linear operators. If $\O$ is a smooth bounded domain, we already proved in \cite{NV} that the generalized principal eigenvalue coincides with the (classical) first eigenvalue of $\CK_V$. Here we investigate the relation between three types of the generalized principal eigenvalue for quasilinear operator on general smooth domain (possibly unbounded), which plays an important role in the investigation of their asymptotic properties. These results form the basis for the study of the simplicity of the generalized principal eigenvalues, the maximum principle and the spectrum of $\CK_V$. We further discuss applications of the notions by providing some examples.
\end{abstract}

\medskip



\tableofcontents

%%%%%%%%%%%%%%%%%%%%%%%%%%%%%%%%%%%%%%%%%%%%%%%%%%%
%%%%%%%%%%%%%%%%%%%%%%%%%%%%%%%%%%%%%%%%%%%%%%%%%%%%%%%%%%%%%SECTION--INTRODUCTION%%%%%%%%%%%%%%%%%%%%%%%%%%%%%%%%%%%%%%%%%%%%%%%%%%%%%%%%%%%%%%%%%%%%%%%%%%%%%%%%%%%%%%%%%%%%%%%%%%%%%%%%%%%%%%%%%%%%%%

\section{Introduction and Main Results}

The principal eigenvalue is a basic notion associated with elliptic operators and plays a crucial role in the analysis of partial differential equation, especially in 
the study of semilinear elliptic problems. The principal eigenvalue for quasilinear operators is also the subject of intensive research since not only it is a natural extension of that of linear operators but also it allows to bring into light new phenomena which stem from the interesting structure of quasilinear operators. In this paper, we investigate the  \textit{ generalized principal eigenvalue} of the  operator 
\bel{KV}\CK_V[u]:=-\Delta_p u + V u^{p-1},\quad u \geq 0, \ee
in $\Gw \subset\R^N$ (possibly unbounded), where $\Gd_p u =\text{div} (|\nabla u|^{p-2}\nabla u)$ with $p>1$ and $V \in L_{\loc}^\infty(\Gw)$, $\inf_{\O}V>-\infty$. 

If $\Omega$ is a $C^{1,\nu}$ ($0<\nu<1$) bounded domain and $V \in L^\infty(\Gw)$, it is well-known that  the variational problem
\begin{equation} \label{varprob} \gl_V^{\Gw}:=\inf_{\phi \in W_0^{1,p}(\Gw) \sms \{0\}}\frac{\int_{\Gw}(|\nabla \gf|^p + V|\gf|^p)dx}{\int_{\O}|\phi|^pdx} \end{equation}
admits a unique (up to multiplicative constants) positive minimizer $\vgf$ (see, e.g., \cite{DKN},  \cite[Lemma 3]{GS}). Moreover, $\vgf \in C^{1,\gth}$ ($0<\gth<1$) and it is a positive solution of the quasilinear eigenvalue problem
\bel{eigenprob} \left\{ \BAL \CK_V [\vgf] &= \gl_V^{\Gw}\vgf^{p-1} \qquad &&\text{in } \Gw \\
\phantom{-,,}
\vgf &= 0 &&\text{on } \prt \Gw. 
\EAL \right. \ee
Here $\gl_V^{\Gw}$ and $\vgf$ are called respectively the \emph{principal eigenvalue} and \emph{eigenfunction} of $\CK_V$ in $\Gw$. Note that since $C_c^1(\O)$ is dense in $W_0^{1,p}(\O)$ with respect to $W^{1,p}$ norm, the infimum in (\ref{lambda}) can be taken over $C_c^1(\O)$.

Problem \eqref{eigenprob} has received much attention in the literature because it has various applications, most of which arise from problems in fluid dynamics, where the p-Laplacian operator with $p\neq 2$ is employed to study non-Newtonian fluids ($p > 2$ for dilatant fluids and $p < 2$ for pseudoplastic fluids). This kind of problem has also been used to develop noise reduction and edge detection techniques in image processing (see \cite{CLMC}), where the degenerate diffusion term enables to smoothen the image without destroying the edges.

When $\Gw$ is a general (possibly unbounded) domain, we introduced a notion of generalized principal eigenvalue of $\CK_V$ in $\Gw$ \cite{NV}
\begin{definition}  \label{lambda} (i) The quantity 
\bel{gpe0} \gl(\CK_V,\O):=\sup \{\lambda\in\R |\, \, \exists\psi \in W_{\loc}^{1,p}(\Gw),\psi>0, \CK_V[\psi] \geq \lambda \psi^{p-1} ~ \text{{ \textit{in the weak sense}} in } \Gw \} \ee
is called a {\em generalized principal eigenvalue} of $\CK_V$ in $\Gw$. 
Here, the inequality holds  \textit{in the weak sense in $\O$}  means
\begin{equation}\label{weaksense}
\int_{\O}|\nabla\psi|^{p-2}\nabla\psi\cdot\nabla\phi \, dx+\int_{\O}V\psi^{p-1}\phi\, dx\geq \lambda\int_{\O}\psi^{p-1}\phi \,dx\quad\quad\forall\phi\in C^\infty_c(\O).
\end{equation}
The functions $\psi$ in \eqref{gpe0} are called {\em admissible test functions} for $\gl(\CK_V,\Gw)$. \smallskip

(ii) We say that $\gl \in \BBR$ is an eigenvalue of $\CK_V$ in $\Gw$ if there exists a positive weak solution $u \in W_{loc}^{1,p}(\Gw)$ of
\bel{eigenRN} \CK_V[u]=\gl u^{p-1} \quad \text{in } \Gw. \ee
Such a solution $u$ is called an eigenfunction of $\CK_V$ associated with $\gl$. Denote by $\CE(\O)$ the set of all eigenvalues of $\CK_V$ in $\O$.
\end{definition}
An important feature of the notion of generalized principal eigenvalue is that if $\Gw$ is a smooth bounded domain, $\gl(\CK_V,\Gw)$ coincides with the principal eigenvalue $\gl_V^\Gw$, while if $\Gw$ is unbounded $\gl(\CK_V,\Gw)$ is well defined and can be expressed by a variational formula. 

This type of eigenvalue is of purely mathematical interest since it is an effective tool in the study of many problems. 

Indeed, the role of $\gl(\CK_V,\Gw)$ is clearly described in the analysis of equation 
\bel{nonlinearR} \CK_V[u] + b \,g(u)=0 \quad \text{in } \Gw \ee
where $0 \leq b \in L^\infty(\Gw)$ and $t \mapsto g(t)/t^{p-1}$ is increasing. We refer the reader to \cite{CDG, DuGu} for the case when $\Gw$ is bounded and to \cite{NV} for the case $\Gw=\BBR^N$. 
In particular, in \cite{NV}, under the assumption on the asymptotic behavior of $V$ near infinity
$$ \liminf_{|x| \to \infty}|x|^q V(x)>0 \quad \text{for some } q \in [0,p],\, p>1, $$
we have proved that:
\begin{theorem} \label{JFA} (See \cite[Theorem 1.3 and Theorem 1.44]{NV})
	
\noindent {\sc I. Existence and Uniqueness.} If $\gl(\CK_V,\Gw)<0$ then there exists a unique decaying solution of \eqref{nonlinearR}. Moreover, 

(i)\, If $q \in [0,p)$ then the unique solution decays exponentially

(ii)\, If $q=p$ then the solution decays polynomially. \smallskip

\noindent {\sc II. Nonexistence.} If $\gl(\CK_V,\Gw) \geq 0$ then there exists no decaying solution of \eqref{nonlinearR}. 
\end{theorem}

\noindent We emphasize that when $p \geq 2$, the existence, uniqueness and nonexistence results hold in a much larger class of functions, including bounded functions. For  more details, we refer the reader to \cite{NV}.

It is noteworthy that the notion of generalized principal eigenvalue in Definition \ref{lambda} is closely related to \textit{the best constant in the Hardy-type inequality} which was introduced by Pinchover et al. to establish optimal Hardy-type inequalities (see  \cite{DFP,DePi}). This notion is also used to study the structure of positive solution homogeneous equation $\CK_V[u]=0$ in unbounded domains (see, e.g., \cite{FO,Murata, Pinchover}). Moreover, it is directly related to the characterization of the Liouville type result and the maximum principle \cite{Pinchover1, BDPR, GS, BNV,QS,BR3}. Therefore, the investigation of the principal eigenvalue is a crucial ingredient to deal with many fundamental questions in the theory of partial differential equations. 

The aim of the present paper is to bring the eigentheory for quasilinear
operators closer to the level of the well-studied linear case  (see \cite{BNV,BR0,BR3}) by establishing qualitative properties, the simplicity, the spectrum of  $\CK_V$ and the maximum principle.  To this end, in the spirit of the papers \cite{BNV,BR0,BR3}, we introduce  other notions of the generalized principal eigenvalue as follows
\begin{definition}\label{gpe} Let 
 $\Gw$ be a (possibly unbounded) domain in $\R^N$. Define
$$ \BAL
&\gl'(\CK_V,\Gw):=\inf \{\lambda\in\R |\, \, \exists\psi \in W_{0}^{1,p}(\Gw),\psi>0, \CK_V[\psi] \leq \lambda \psi^{p-1} ~ \text{{ \textit{in the weak sense}} in } \Gw\},\\
 &\gl''(\CK_V,\Gw):=\sup \{\lambda\in\R |\, \, \exists\psi \in C_{\loc}^{1}(\Gw), \inf_{\O}\psi>0, \CK_V[\psi] \geq \lambda \psi^{p-1} ~ \text{{ \textit{in the weak sense}} in } \Gw \},
\EAL $$
 where the inequalities are understood in the weak sense in $\O$ as in \eqref{weaksense}. The functions $\psi$ in \eqref{gpe0} are called {\em admissible test functions}.
\end{definition}

Note that since $\inf_{\Gw}V>-\infty$, it follows that $\gl(\CK_V,\Gw)>-\infty$ and $\gl''(\CK_V,\Gw)>-\infty$. It may occur that the set of admissible test functions in the definition of $\gl'(\CK_V,\Gw)$ is empty. In such a case, we set $\gl'(\CK_V,\Gw)=+\infty$.

The main difference between our notions of the \textit{generalized  principal eigenvalues} and those introduced by Berestycki et al. \cite{BNV,BR0,BR3} is that in Definitions \ref{lambda} and \ref{gpe} admissible test functions $\psi$ are only required to be in $W^{1,p}$ or in $C_{loc}^1$ and all the inequalities are only required to hold  \textit{in the weak sense in} $\Gw$ while admissible test functions in the definitions of the generalized principal eigenvalue in  \cite{BNV,BR0,BR3} belong to $W^{2,N}_{loc}$ and the inequalities are understood \textit{almost everywhere in} $\Gw$. Moreover, our definition of $\lambda'(\CK_V,\Gw)$  differs from the one defined by Berestycki and Rossi for second order linear operators in \cite{BR3} in the sense that we impose  admissible test functions $\psi$ to be in $W^{1,p}_0(\O)$ instead of requiring them to be in $W_{loc}^{2,N}(\Gw)$ and to satisfy  $\lim_{x\to\zeta}\psi(x)=0$ for every $\zeta \in \prt \Gw$. Despite these differences, our notions fit well into the framework of quasilinear operators and allow to obtain the properties of the generalized principal eigenvalue. We emphasize that, many new ideas have been developed in this paper to overcome the fundamental difficulties stemming from the nonlinearity of p-Laplacian since most of the techniques used in  \cite{BNV,BR0,BR3} fail to apply in this framework, especially to obtain the relation of the notions of the generalized principal eigenvalue, the simplicity and the maximum principle.

Observe, by Definitions \ref{lambda} and \ref{gpe}, that $\gl''(\CK_V,\Gw) \leq \gl(\CK_V,\Gw)$. 
Our first result reveals a more profound relation between three notions of generalized principal eigenvalues.
 \begin{theorem}\label{equivalence} Let $\O\subset\R^N$. Assume $0 \not \equiv V \in L^\infty(\Gw)$. 
 
\noindent  (i)\, If $\O$ is a smooth bounded domain then
$$\gl(\CK_V,\Gw)=\gl'(\CK_V,\Gw) =\gl''(\CK_V,\Gw).$$

\noindent (ii)\, If $\O=\R^N$ then
$$\gl(\CK_V,\Gw) = \gl'(\CK_V,\Gw) \geq \gl''(\CK_V,\Gw).$$
 
\noindent (iii)\, If $\O=\R^N$ and $V$ is radially symmetric then
 $$\gl(\CK_V,\Gw) = \gl'(\CK_V,\Gw) =\gl''(\CK_V,\Gw).$$ 
 \end{theorem}
The following result describes the effect of the diffusion coefficient on the generalized principal eigenvalue.
 
 \begin{theorem} \label{lim1} Let $\Gw$ is a (possibly unbounded) domain in $\BBR^N$ and $V \in L_{loc}^\infty(\Gw)$. For $\ga>0$, denote
 	\bel{La} \CL_\alpha[\phi]:=-\alpha\Delta_p\phi+V\phi^{p-1},\quad\phi\geq 0.\ee
 	Then the following properties hold.
 	
\noindent (i)\,  The mapping $\ga \mapsto \gl(\CL_\ga,\Gw)$ is concave, nondecreasing and
 	$$\lim_{\alpha\to0}\lambda(\CL_\alpha,\Gw)=\inf_\Gw V.$$
 	
\noindent (ii)\, If $\Gw$ is a smooth bounded domain then
 	$$\lim_{\alpha\to+\infty}\lambda(\CL_\alpha,\Gw)=\infty.$$
 	
\noindent (iii)\, If $\Gw=\R^N$  then

 	\bel{limsup} \limsup_{\alpha\to+\infty}\lambda(\CL_\alpha,\BBR^N)\leq\limsup_{{|x|\to+\infty}} V(x), \ee
 	
 	\bel{liminf}
 	\liminf_{\alpha\to+\infty}\lambda(\CL_\alpha,\BBR^N)\geq\liminf_{{|x|\to+\infty}} V(x).
 	\ee
 \end{theorem}
 
Theorem  \eqref{lim1} yields different interesting phenomena  of the effect of diffusion coefficient on the generalized principal eigenvalue between bounded and unbounded domains. From (i) and (ii) and Theorem \ref{JFA}, one sees that if $\O$ is bounded then for any bounded potential $V$ with negative infimum, equation \eqref{nonlinearR}  admits a positive solution for $\alpha$ small and does not admit any positive solution for $\alpha$ large. The phenomenon is strikingly different when $\O=\R^N$. Assume there exists $\lim_{|x|\to\infty}V(x)=\ell$. From \eqref{limsup}) and \eqref{liminf}, if $\ell<0$, then
$$\lim_{\alpha\to+\infty}\lambda(\CL_\alpha,\BBR^N)=\ell<0,$$
and  Theorem \ref{JFA} implies that equation (\ref{nonlinearR}) admits a unique  positive solution for $\alpha$ large while if $\ell>0$ equation (\ref{nonlinearR}) admits no positive solution for $\alpha$ near $+\infty$.

Thanks to Theorems \ref{equivalence} and \ref{lim1}, we obtain a result  about the effect of the potential on the generalized principal eigenvalues.

If $\{V_\ga\}$ is a sequence of functions in $L_{loc}^\infty(\Gw)$ we denote
\bel{Ka} \CK_\ga[\phi]:=-\Delta_p\phi+V_\ga\,\phi^{p-1}, \quad \gf \geq 0. \ee

\begin{theorem}\label{lim2} 
Let $\Gw$ be a (possibly unbounded) domain and $V \in L_{loc}^\infty(\Gw)$.

\noindent (i)\,  Denote $V_\alpha(x)=V(\alpha x)$ for $\ga>0$. If $\Gw=\BBR^N$ and $V(0)=\inf_{\R^N}V$ then 
\bel{lubu} \lim_{\alpha\to 0}\lambda(\CK_\alpha,\BBR^N)=\inf_{\BBR^N}V \ee
and if, in addition,  $\sup_{\R^N}V=\liminf_{|x|\to\infty}V(x)$ then
\bel{lubu1} \lim_{\alpha\to +\infty}\lambda(\CK_\alpha,\BBR^N)=\sup_{\R^N}V. \ee

\noindent (ii)\, Denote  $V_\alpha(x)=\alpha V(x)$ for $\ga>0$. If $V$ is upper semi-continuous then
\bel{upper} \lim_{\alpha\to+\infty}\frac{\lambda(\CK_\alpha,\Gw)}{\alpha}=\inf_{\O}V.\ee

\noindent (iii)\, Denote  $V_\alpha(x)=-\alpha V(x)$ for $\ga>0$. If $V$ is lower semi-continuous then
\bel{lower} \lim_{\alpha\to+\infty}\frac{\lambda(\CK_\alpha,\Gw)}{\alpha}=\sup_{\O}V.\ee
\end{theorem}

An example of a potential $V$ which inspires the study of the limits of $\{\gl(\CK_\ga,\Gw)\}$ is given by
$$
V(x)=\left\{ \BAL
&-e^{-\frac{|x|^2}{1-|x|^2}}+\frac{1}{2} \quad &&\text{if } |x|<1\\
&\frac{1}{2} &&\text{if } |x|\geq1.
\EAL \right.
$$
For such $V$, we see that 
$$\min_{\R^N}V(x)=V(0)=-\frac{1}{2} \quad \text{and} \quad \max_{\R^N}V(x)=\lim_{|x| \to +\infty}V(x)=\frac{1}{2}.$$ 
If we put $V_\alpha(x)=V(\alpha x)$ then $V_\alpha$ is continuous and increasing with respect to $\alpha$. By Theorem \ref{lim2} (i), for $\alpha$ small enough, $\lambda(\CK_\ga,\BBR^N)<0$ and $\lambda(\CK_\ga,\BBR^N)>0$ for $\alpha$ large enough. Hence in view of Theorem \ref{JFA} (as $q=0$), there exists a threshold value $\alpha^\star$ such that equation \eqref{nonlinearR} admits a unique positive solution if and only if $\alpha<\alpha^\star$. If we put $V_\alpha(x)=\alpha V(x)$ then by Theorem \ref{lim2} (ii) and (iii), equation \eqref{nonlinearR} admits a unique positive solution as $\alpha$ near $+\infty$ and admits no positive solution as $\alpha$ near $-\infty$.

Next we discuss the simplicity of $\gl(\CK_V,\Gw)$. For this purpose,   we introduce the notion of \textit{solutions of minimal growth at infinity} in the spirit of \cite[Definition 8.2]{BR3}.

\begin{definition} \label{defminimal} Assume $\Gw$ is a unbounded domain and $V \in L_{loc}^\infty(\Gw)$. We say that a positive weak solution $u\in W^{1,p}_{loc}(\Gw)$ of
	\bel{eqKV}
	\CK_V[u]=0\qquad \text{in } \Gw,
	\ee
	is a solution of \eqref{eqKV} of minimal growth at infinity if  for any $\rho>0$ and any positive function $v\in C^{1}_{loc}(\Gw \setminus B_\rho)$ satisfying $\CK_V[v]\geq0$ in the weak sense in $\Gw \setminus B_\rho$, there exist $\rho'>\rho$ and $k>0$ such that $u\leq k v$  in $\Gw \setminus B_{\rho'}$.
\end{definition}

In the sequel, we treat the case $\Gw=\BBR^N$. We know that if $u$ is a positive eigenfunction of $\CK_V$ in $\BBR^N$ associated to $\gl(\CK_V,\BBR^N)$ then $u$ is a positive weak solution of 
\begin{equation}\label{eigen}
\CK_V[u]=\lambda(\CK_V,\R^N)u^{p-1}\quad \text{in } \BBR^N.
\end{equation}
\begin{theorem} \label{simplicity1} Assume $V \in L^\infty(\BBR^N)$ satisfies
	\bel{condV} \limsup_{|x| \to \infty}|x|^{p-1}|V(x)-\gl(\CK_V,\BBR^N)|<\infty. \ee
	If there exists a positive eigenfunction $u \in W_{loc}^{1,p}(\BBR^N)$ of $\CK_V$ associated to $\gl(\CK_V,\BBR^N)$ such that 
	\begin{align}
	&u \text{ is a solution of minimal growth solution of } \eqref{eigen}, \\
	&\nabla u(x) \ne 0 \quad \forall x \in \BBR^N, \label{nablau}\\
    &\liminf_{|x| \to \infty}\frac{|x||\nabla u(x)|}{u(x)}>0, \label{nablauu}
	\end{align}
then $\gl(\CK_V,\BBR^N)$ is simple, i.e. if $v \in W_{loc}^{1,p}(\BBR^N)$ is a positive eigenfunction of $\CK_V$ in $\BBR^N$ associated with $\gl(\CK_V,\BBR^N)$ then $v=\ell u$ in $\BBR^N$ for some $\ell>0$. 
\end{theorem}

Let us remark that  the proof of Theorem \ref{simplicity1} is mainly based on the strong comparison principle \cite[Theorem 3.2]{FrPi} and the scaling technique. Assumption \eqref{nablau} is needed to make use of the strong comparison principle and is imposed in many papers, for instance in  \cite[Theorem 3.2]{FrPi}, \cite[Theorem 3.4]{BBV} and in a series of celebrated papers by Pucci and Serrin \cite[Lemma 5]{PuSe1}, \cite[Theorem 10.1]{PuSe2}, \cite[Theorem 4]{PSZ}. Assumption \eqref{condV} means that $V$ is not allowed to be far away from $\gl(\CK_V,\BBR^N)$ as $|x|$ near infinity. This assumption and \eqref{nablauu} are employed in the scaling process in order to treat the case when the graph of $u$ and the graph of $\ell v$ are tangent at infinity where $\ell$ is some positive constant.  

When $p \geq 2$, we obtain the simplicity of $\gl(\CK_V,\BBR^N)$ without making use of the notion of solution of minimal growth at infinity. Unlike assumption \eqref{condV}, in this case, we impose the condition that $V$ is sufficiently far from $\gl(\CK_V,\BBR^N)$ as $|x|$ near $\infty$. This enables us to employ Proposition \ref{class1} in order to deduce that any eigenfunction associated to $\gl(\CK_V,\BBR^N)$ decays exponentially. Consequently, by adapting the classical method \cite{DKN,DrRa}, we obtain the simplicity.
\begin{theorem} \label{simplicity2}
	Assume $p \geq 2$ and $V\in L^\infty(\R^N)$. There exists $\mu>0$ such that if
	\begin{equation}\label{decay}
	\lambda(\CK_V,\R^N)<\liminf_{|x|\to\infty}V(x)-\mu
	\end{equation}
	then $\lambda(\CK_V,\R^N)$ is simple. Moreover, the unique  (up to multiplicative constants) eigenfunction  associated with  $\lambda(\CK_V,\R^N)$ decays exponentially.
\end{theorem}

%We give an example to demonstrate the significance of Theorem \ref{simplicity2}. Indeed, by Theorem \ref{lim2} (i), if we choose $V$ such that 
%\begin{equation}\nonumber
%	\inf_{\R^N}V<\liminf_{|x|\to\infty}V(x)-\mu
%	\end{equation}
%for some $\mu>0$ large enough, then one sees that $\lambda(\CK_\alpha,\R^N)\to\inf_{\R^N}V$ as $\alpha\to0$. Obviously, (\ref{decay}) is fulfilled and Theorem \ref{simplicity2} implies that  $\lambda(\CK_\alpha,\R^N)$ is simple as $\alpha$ small enough.

It is well known that if $\Gw$ is a smooth bounded domain then $\gl(\CK_V,\Gw)$ is isolated (see \cite{GS}), i.e. $\CE(\Gw)=\{ \gl(\CK_\ga,\Gw) \}$. This property no longer holds if $\Gw$ is unbounded. This is reflected in the following result.

\begin{theorem} \label{spectrum} Assume $V \in L^\infty(\BBR^N)$. Then
	\bel{spec}  \CE(\BBR^N)=(-\infty,\gl(\CK_V,\BBR^N)]. \ee
\end{theorem}
It is noteworthy that Theorem \ref{spectrum} still holds true if $\BBR^N$ is replaced by a smooth unbounded domain. It can be obtained by using an analogue argument as in the proof of Theorem \ref{spectrum}. However,  we state the result for the case $\Gw=\BBR^N$ in order to simplify the proof and to streamline the exposition. 

We end the section by providing a criterion in terms of $\gl''(\CK_V,\BBR^N)$ to characterize the maximum principle in $\R^N$.
\begin{theorem} \label{MP} (Weak maximum principle) Assume $\gl''(\CK_V,\R^N) >0$. If a function  $u\in  C^1_{loc}(\R^N)$ satisfies  
\begin{equation}\nonumber
 \left\{ \BAL 
  & \CK_V[u] \leq 0 \quad \text{in the weak sense in } \BBR^N, \\
	& \sup_{\BBR^N}u<\infty, \quad \limsup_{|x| \to \infty}u(x) \leq 0, \\
	& |\nabla u(x)| \neq 0,\quad\forall x \in \BBR^N,
		\EAL \right.
\end{equation}	
	then $u \leq 0$ in $\BBR^N$. 
\end{theorem}
As a consequence of Theorem \ref{equivalence} (iii) and Theorem \ref{MP}, we obtain
\begin{corollary} Assume $0 \not \equiv V \in L^\infty(\BBR^N)$ is radially symmetric. Then the weak maximum principle holds if $\gl(\CK_V,\R^N) >0$. 

\end{corollary}

The paper is organized as follows. In Section 2, we prove Theorem \ref{equivalence}. The proof of Theorem \ref{lim1} and Theorem \ref{lim2} are presented in Section 3. Finally, in Section 4, we demonstrate Theorems \ref{simplicity1}, \ref{simplicity2}, \ref{spectrum} and \ref{MP}. \medskip

\noindent \textbf{Notation.} Throughout the paper, $B_r$ denotes the ball of center $0$ and radius $r>0$. Unless otherwise stated, we assume that $p>1$ and $V \in L^\infty(\Gw)$.

\section{Equivalence of the generalized principal eigenvalues}
This section is devoted to the study of the relation between the notions  $\gl(\CK_V,\Gw)$, $\gl'(\CK_V,\Gw)$ and $\gl''(\CK_V,\Gw)$. 
%Denote
%$$ \Gl_V(\Gw):=\{\lambda\in\R |\, \exists\psi \in W_{\loc}^{1,p}(\Gw),\psi>0, \CK_V[\psi] %\geq \lambda \psi^{p-1} ~ \text{ in the weak sense in } \Gw \}. $$
%Note that $\Gl_V(\Gw) \not = \emptyset$ since $\inf_{\BBR^N}V>-\infty$.
Let us first recall the following result, which is proved in \cite{NV}
 
 \begin{theorem}\label{main0} 
(1)\, Assume $\Gw$ is a $C^{1,\nu}$ ($\nu\in(0,1)$) bounded domain in $\BBR^N$ and $V \in L^\infty(\Gw)$. Then 
$$\gl(\CK_V,\O)=\gl_V^\Gw(\Gw).$$ 

\noindent (2)\, Assume $\Gw$ is a general domain in $\BBR^N$ (possibly unbounded) and $\{\Gw_n\}$ is a smooth  exhaustion of $\Gw$. Let $V \in L^\infty_{\loc}(\Gw)$ such that $\inf_{\O}V>-\infty$. Then the following properties hold. \medskip

(i)\, $\inf_{\O}V \leq \gl(\CK_V,\O_{n+1})< \gl(\CK_V,\O_{n})$ for every $n \in \BBN$. \medskip

(ii)\, $\gl(\CK_V,\O)=\lim_{n\to\infty}\gl(\CK_V,\O_n)$  and there exists a positive weak solution $\vgf \in C^1_{\loc}(\Gw)$ of 
$$ \CK_V[\vgf]=\gl(\CK_V,\Gw)\vgf^{p-1}\quad\quad\textrm{in $\O$.} $$ 

(iii)\, 
\bel{lamV}
\gl(\CK_V,\O)=\inf_{\phi \in C_c^{1}(\Gw) \sms\{0\}}\frac{\int_{\Gw}(|\nabla \gf|^p + V|\gf|^p)dx}{\int_{\O}|\phi|^pdx}.
\ee
\end{theorem}
 
\begin{lemma}\label{la2}  Let $\Gw$ be a smooth bounded domain or $\Gw=\BBR^N$. Assume $0 \not \equiv V \in L^\infty(\Gw)$. Then there holds
\bel{lub1} \gl(\CK_V,\Gw) \geq \gl'(\CK_V,\Gw). \ee 
\end{lemma}

\Proof In order to prove \eqref{lub1}, we need to show that $\gl \geq \gl'(\CK_V,\Gw)$ for any $\gl > \gl(\CK_V,\Gw)$. Without lost of generality we assume that $\gl=0$. Since $\gl(\CK_V,\Gw)<0$, by Theorem \ref{main0}, (2.ii), there is a smooth bounded domain $G \Subset \Gw$ such that $\gl(\CK_V,G)<0$ and $V \not \equiv 0$ in $G$. Let $\vgf_V^{G}$ be the first eigenfunction associated to $\gl(\CK_V,G)<0$, normalized by
\bel{norm} \max_{G}\vgf_V^G = \min\left \{ 1,-\frac{\gl(\CK_V,G)}{\sup_{G}{|V|}}  \right \}. \ee
In particular, the function $\vgf_V^G$ satisfies
\bel{eigenG} \left\{ \BAL \CK_V [\vgf_V^G] &= \gl(\CK_V,G) (\vgf_V^G)^{p-1} \qquad &&\text{in } G \\
\vgf_V^G &= 0 &&\text{on } \prt G. 
\EAL \right. \ee

\medskip

\noindent \textit{Case 1: $\Gw$ is a smooth bounded domain.}

Define $\ovl \vgf: = 1$ and 
$$ \unl \vgf: = \left\{  \BAL &\vgf_V^G \quad &&\text{in } G, \\
&0 &&\text{in } \Gw \sms G. \EAL \right.
$$
It can be verified that $\ovl \vgf$ and $\unl \vgf$ are respectively weak supersolution and weak subsolution of 
\bel{lpe} \left\{ \BAL \CK_V [u]  + V^+ u^p &= 0 \qquad &&\text{in } \Gw \\
u &= 0 &&\text{on } \prt \Gw, 
\EAL \right. \ee
where $V^+=\max\{V,0\}$. By \cite[Theorem 3.1]{LS}, one can find a solution $u$ of \eqref{lpe} such that $\unl \vgf \leq u \leq \ovl \vgf$ in $\Gw$. Let $x_0 \in G$ such that $\unl \vgf(x_0)=\max_G \unl \vgf >0$. Then $u(x_0) \geq \unl \vgf(x_0) >0$. By Harnack inequality \cite{Se,Tr}, we deduce that $u>0$ in $\Gw$. \medskip

\noindent \textit{Case 2: $\Gw=\BBR^N$.} In this case, we can take $G=B_{R_0}$ for some $R_0>0$ large and $\gl(\CK_V,B_{R_0})<0$. Denote by  $\vgf_V^{B_{R_0}}$ the eigenfunction associated to $\gl(\CK_V,B_{R_0})$. Hence \eqref{norm} and \eqref{eigenG} hold with $G$ replaced by $B_{R_0}$.

Fix $m>0$ and set
$$a_{R_0,m}(x):=\left\{\BAL
 &-V(x) \quad &&x \in B_{R_0}\\
 &-\max\{V(x), m\}&& x \in B_{R_0}^c.
\EAL \right.$$
It is easy to see that
$$ \limsup_{|x| \to \infty}a_{R_0,m}<-m $$
and that there exists $s_0 > 1$ independent of $R_0$ and $m$ such that 
$$ -a_{R_0,m}s_0^{p-1} + V^+s_0^p \geq 0.$$
Consider the equation
\bel{lomo} -\Delta_p u - a_{R_0,m} u^{p-1} + V^+ u^{p} = 0 \quad\quad \text{in }\R^N. \ee
By \cite[Proposition 4.1]{NV}, we deduce that the following function
$$ \ovl v(x):=s_0 \chi_{B_{R_0}}(x) + Ce^{-\gth|x|}\chi_{B_{R_0}^c}(x),\quad x \in \BBR^N,$$
with $\gth=\gth(p,m)$ and $C=C(s_0,m,p)$, is a decaying supersolution of \eqref{lomo}.  

For $R>R_0$, define
$$ \unl u_R: = \left\{  \BAL &\vgf_V^{B_{R_0}} \quad &&\text{in } B_{R_0}, \\
&0 &&\text{in } B_R \sms B_{R_0}. \EAL \right.
$$
We can see that $\unl u_R$ is a subsolution of \eqref{lomo}. Indeed, in $B_{R_0}$
\bel{lomosub} \BAL -\Delta_p \unl u_R - a_{R_0,m} \unl u_R^{p-1} + V^+ \unl u_R^p &= -\Delta_p \vgf_V^{B_{R_0}} + V  (\vgf_V^{B_{R_0}})^{p-1} + V^+ (\vgf_V^{B_{R_0}})^p \\
&= \gl_V^{B_{R_0}} (\vgf_V^{B_{R_0}})^{p-1} + V^+(\vgf_V^{B_{R_0}})^p \\
&\leq (\vgf_V^{B_{R_0}})^{p-1}(\gl_V^{B_{R_0}} + V^+\vgf_V^{B_{R_0}}) \\
&\leq 0.
\EAL \ee 
Here the last inequality follows from the normalization of $\vgf_V^{B_{R_0}}$.

It is easy to see that $\unl u_R$ and $\ovl v$ are respectively sub and supersolution of 
\bel{lpR} \left\{ \BAL -\Gd_p u -a_{R_0,m}u^{p-1} + V_+ u^p &= 0 \qquad &&\text{in } B_R \\
u &= 0 &&\text{on } \prt B_R
\EAL \right. \ee
such that $\unl u_R \leq \ovl v$ in $B_R$. Therefore, by sub-supersolutions theorem, we deduce the existence of a solution $u_R$ of \eqref{lpR} in $B_R$ such that $\unl u_R \leq u_R \leq \ovl v$ in $B_R$. By regularity result for quasilinear elliptic equations, up to a subsequence, $\{ u_R \}$ converges in $C_{\text{loc}}^1(\BBR^N)$, as $R \to \infty$, to a weak  solution $u^*$ of \eqref{lomo} in $\BBR^N$. Moreover, $u^*(x_0) \geq \vgf_V^{B_{R_0}}(x_0)>0$ and $0 \leq u^* \leq \ovl v$ a.e. in $\BBR^N$. By Harnack inequality we infer that $u^*>0$ in $\BBR^N$. Since $\ovl v$ decays exponentially, so does $u^*$. By local regularity for quasilinear elliptic equations (see \cite{Tol} and \cite[Theorem 3.1]{La}) and Harnack inequality \cite{Tr}, there exists constants $c_1,c_2$ depending on $s_0$, $N$, $p$, $\norm{V}_{L^\infty(\BBR^N)}$, $m$ such that for every $x \in \BBR^N$
$$ \sup_{B_1(x)}|\nabla u^*| \leq c_1 \sup_{B_2(x)} u^* \leq c_2 \inf_{B_2(x)} u^*. $$
It follows that $u^* \in W_0^{1,p}(\BBR^N)$. 

Since $a_{R_0,m} \leq -V$, we deduce that
$$ \CK_V[u^*] = (V+a_{R_0,m})(u^*)^{p-1} - V^+ (u^*)^p \leq 0.$$
By choosing $\psi=u^*$ in the definition of $\gl'(\CK_V,\Gw)$ we deduce that $\gl'(\CK_V,\Gw) \leq 0$. This completes the proof.  \qed

\begin{lemma} \label{la3} Assume $\Gw$ is a general domain in $\BBR^N$ and $V \in L^\infty(\Gw)$. Then there holds 
\bel{lamlam} \lambda(\CK_V,\O)\leq \lambda'(\CK_V,\O). \ee
\end{lemma}
\Proof Take $\lambda\in\R$ such that  there exists an admissible test function $\psi\in W^{1,p}_{0}(\O)$ satisfying
$\psi>0$ in $\O$ and 
$$ \CK_V[\psi]\leq\lambda\psi^{p-1} \quad \text{in the weak sense in } \Gw.$$
we will prove that $\lambda(\CK_V,\O)\leq\lambda$. 

Since $\psi\in W_0^{1,p}(\O)$ and $C_c^1(\Gw)$ is dense in $W_0^{1,p}(\Gw)$, there exists a sequence $\{\psi_n\} \sbs C_c^1(\Gw)$ converging to $\psi$ in $W^{1,p}(\Gw)$. Since $\psi>0$, we infer that $\psi_n>0$ for $n$ large enough. By \eqref{lamV}, we have
$$ \lambda(\CK_V,\Gw)\leq\frac{\int_{\Gw}(|\nabla \psi_n|^{p}+V \psi_n^p)dx}{\int_{\Gw}\psi_n^pdx}.
$$
Letting $n\to \infty$ yields
$$
\lambda(\CK_V,\Gw)\leq\frac{\int_{\Gw}(|\nabla \psi|^{p}+V \psi^p)dx}{\int_{\Gw}\psi^pdx}= \lambda.
$$
This completes the proof. \qed 

\begin{proposition} \label{equal1} Let $\Gw$ be a smooth bounded domain or $\Gw=\BBR^N$. Assume $0 \not \equiv V \in L^\infty(\Gw)$. There holds
\bel{lambda01}\lambda(\CK_V,\O)=\lambda'(\CK_V,\O) \ee
\end{proposition}
\Proof Equality \eqref{lambda01} follows directly from Lemma \ref{la2} and Lemma \ref{la3}. \qed

\begin{lemma} \label{cutoff} Let $\gl \in \BBR$ and $V \in L_{loc}^\infty(\BBR^N)$. If $0 \leq u \in C_{loc}^1(\BBR^N)$ satisfies $\CK_V u \leq \gl u^{p-1}$ in the weak sense in $\BBR^N$ then 
\bel{estc} (\gl(\CK_V,\BBR^N)-\gl)\int_{\BBR^N} u^p \psi^p dx \leq C(p)\int_{\BBR^N} u^p |\nabla \psi|^p dx \quad \forall  \psi \in C_c^1(\BBR^N), \, \psi \geq 0.\ee
\end{lemma}
\begin{proof}From the assumption, we have
\bel{testfor} \int_{\BBR^N}(|\nabla u|^{p-2}\nabla u \nabla \gf + V u^p \gf)dx \leq \gl \int_{\BBR^N}u^{p-1}\gf \,dx \quad \forall \gf \in C_c^1(\BBR^N).
\ee
By choosing $\gf=u\psi^p$ as a test function in \eqref{testfor}, we get
$$ \int_{\BBR^N}(|\nabla u|^p \psi^p + pu\psi^{p-1}|\nabla u|^{p-2}\nabla u \nabla \psi + V u^p \psi^p)dx \leq \int_{\BBR^N}Vu^{p}\psi^p \,dx.
$$
It follows that
$$ \int_{\BBR^N}(|\nabla u|^p \psi^p +  V u^p \psi^p)dx \leq p\int_{\BBR^N}u\psi^{p-1}|\nabla u|^{p-1}|\nabla \psi|\, dx + \gl \int_{\BBR^N}u^{p}\psi^p \,dx.
$$ 
By Young's inequality, we deduce that 
$$ \int_{\BBR^N}(|\nabla u|^p \psi^p +  V u^p \psi^p)dx \leq C(p)\Big(\int_{\BBR^N}u^p |\nabla \psi|^p\, dx + \gl \int_{\BBR^N}u^{p}\psi^p \,dx\Big).
$$ 
This, together with the following estimate
$$ |\nabla(u\psi)|^p dx \leq C(p)\Big( |\nabla u|^p \psi^p + u^p |\nabla \psi|^p  \Big) $$
leads to
$$ \int_{\BBR^N}(|\nabla (u\psi)|^p  +  V (u\psi)^p)dx \leq C(p)\Big(\int_{\BBR^N}u^p |\nabla \psi|^p\, dx + \gl \int_{\BBR^N}(u\psi)^p \,dx\Big).
$$ 
This, joint with Theorem \ref{main0}, implies \eqref{estc}.
\end{proof}

\begin{proposition} \label{equal2} (i) If $\O$ is a smooth bounded domain and $V \in L^\infty(\Gw)$ then 
\bel{equali1} \lambda''(\CK_V,\O)=\lambda(\CK_V,\O).\ee

\noindent (ii) Assume $\Gw$ is a smooth bounded domain or $\Gw=\BBR^N$. If $0 \not \equiv V \in L^\infty(\Gw)$ then
\bel{equali2} \lambda''(\CK_V,\O)\leq \lambda'(\CK_V,\O) .\ee
 
\noindent (iii) If $\Gw=\BBR^N$ and $V \in L_{loc}^\infty(\BBR^N)$ is radially symmetric then \bel{equali3} \gl(\CK_V,\BBR^N)=\gl''(\CK_V,\BBR^N). \ee

\end{proposition}

\Proof
i) Obviously, $\lambda''(\CK_V,\O)\leq\lambda(\CK_V,\O).$ Let us prove that  $\lambda''(\CK_V,\O)\geq\lambda(\CK_V,\O).$ It suffice to show that   $$\lambda \leq \lambda''(\CK_V,\O)$$ for any  $\lambda<\lambda(\CK_V,\O)$. 

Let $\CO$ be a smooth bounded neighborhood of $\partial\O$ then $\Gw \cup \CO$ is a smooth bounded domain. Consider an  extension of $V$ to $\Gw \cup \CO$, still denoted by $V$, such that $V \in L^\infty(\Gw \cup \CO)$. Let $\{\Gw_n\}$ be a decreasing sequence of smooth bounded domains such that
$$\overline{\Gw_1\setminus\Gw}\subset \CO,\quad \ol\Gw\subset\Gw_n \quad \text{for every }n,\quad\quad\bigcap_n\ol{\Gw}_n=\ol{\Gw},$$
and $\Gw_n$ has the same smoothness as $\Gw$ for every $n$. Let $(\lambda_n,\phi_n)$ be the principal eigenvalue and the corresponding eigenfunction of $\CK_V$ in $\Gw_n$ with normalization $\phi_n(x_0)=1$ for some fixed reference point $x_0 \in \Gw$. 

We claim that $$\lim_{n\to\infty}\lambda_n=\lambda(\CK_V,\O).$$
Assume for the moment that the claim holds true. Since $\gl<\gl(\CK_V,\Gw)$, there exists $N_0>0$ large enough such that $\gl_n >\gl$ for every $n>N_0$. Since $\ovl\O\subset\O_n$,  it follows that $\inf_{\ovl\Gw} \phi_n>0$. Choosing $\psi=\phi_n|_{\Gw}$ in the definition of $\lambda''(\CK_V,\O)$ in Definition \ref{gpe}, we obtain that  $\lambda_n\leq\lambda''(\CK_V,\O)$ and thus $\lambda \leq\lambda''(\CK_V,\O)$.

It remains to prove the claim. By Theorem \ref{main0} (2i), we deduce that $\{ \gl_n \}$ is strictly decreasing and therefore there exists $\gl^*=\lim_{n \to \infty}\gl_n$. By the local regularity of weak solutions of degenerate elliptic equations \cite{Di}, up to a subsequence, $\phi_n$ converges  in $C^1_{loc}(\Gw)$ to a function $\phi^*$ which is a weak solution of 
$$ \CK_V [\phi^*]= \lambda^* (\phi^*)^{p-1} \quad \text{in } \Gw,$$ 
Since $\Gw_n$ has the same smoothness as $\Gw$, by the $C^1$ regularity up to the boundary \cite{Lieberman}, we obtain $\phi^*=0$ on $\partial\O$. Note that $\phi^*(x_0)=1$, by the Harnack inequality \cite{Tr}, we get $\phi^*>0$ in $\Gw$ and hence it is a Dirichlet principal eigenfunction of $\CK_V$ in $\O$. This implies $\lambda^*=\lambda(\CK_V,\O)$. \medskip

\noindent (ii) By Proposition \ref{equal1}, $\lambda'(\CK_V,\O)=\lambda(\CK_V,\O)$. This, together with the fact that $\lambda''(\CK_V,\O) \leq\lambda(\CK_V,\O)$, implies inequality \eqref{equali2}. \medskip

\noindent (iii)  It is easy to see that $\gl''(\CK_V,\BBR^N) \leq \gl(\CK_V,\BBR^N)$. It remains to show that $\gl''(\CK_V,\BBR^N) \geq \gl(\CK_V,\BBR^N)$. To this end, take arbitrary $\gl < \gl(\CK_V,\BBR^N)$ and we will demonstrate that $\gl \leq \gl''(\CK_V,\BBR^N)$. Without loss of the generality, we assume that $\gl=0$. This follows that $\gl(\CK_V,\BBR^N)>0$. 

For any $n>0$, put $\gl_n=\gl(\CK_V,B_n)$. By Theorem \ref{main0}, $\gl_n \downarrow \gl(\CK_V,\BBR^N)$ as $n \to \infty$, hence $\gl_n \geq \gl(\CK_V,\BBR^N)>0$. For any $n \in \BBN$, let $f_n \in C^\infty(B_n)$ be nonnegative, radially symmetric and not identically equal to zero in $B_n$ with $\supp f_n \in B_n \sms B_{n-1}$. By \cite[Theorem 2 (v)]{GS}, there exists a unique nonnegative weak solution $w_n \in W_0^{1,p}(B_n)$ of
$$ \left\{  \BAL \CK_V[w_n] &= f_n \quad &&\text{in } B_n \\
w_n &= 0 &&\text{on } \prt B_n.
\EAL \right.
$$
Since $V$ and $f_n$ are both radially symmetric, the uniqueness implies that $w_n$ is radially symmetric too. Moreover, by the strong maximum principle \cite[Theorem 2 (ii)]{GS}, we obtain that $w_n$ is positive in $B_n$. Put 
$$ \vgf_n(x):=\frac{w_n(x)}{w_n(0)} $$
then $\vgf_n$ is radially symmetric and $\vgf_n(0)=1$. By  the Harnack inequality \cite{Tr} and regularity results for quasilinear elliptic equations \cite{Di}, up to a subsequence, the sequence $\{\vgf_n\}$ converges in $C_{loc}^1(\BBR^N)$ to a function $\vgf$ which is a nonnegative, radially symmetric weak solution of $\CK_{V}[\vgf]=0$ in $\BBR^N$. Since $\vgf(0)=1$, in light of the Harnack inequality, we deduce that $\vgf$ is positive in $\BBR^N$.

Next, for $n \in \BBN$, let $\psi_n \in C^\infty(\BBR^N)$ such that  $0 \leq \psi_n \leq 1$, $\psi_n=1$ in $B_{n-1}$, $\psi_n=0$ in $B_n$ and $|\nabla \psi_n| \leq C$ where $C$ is a constant independent of $n$. Applying Lemma \ref{cutoff}, we get
$$ \gl(\CK_V,\BBR^N)\int_{B_{n-1}}\vgf^p dx \leq c\int_{B_n \sms B_{n-1}}\vgf^p dx \quad \forall n \in \BBN $$
where $c$ is independent of $n$. Thus there exist constants $C>0$ and $a>1$ such that 
$$ \int_{B_n \sms B_{n-1}}\vgf^p dx \geq Ca^n \quad \forall n \in \BBN. $$
Therefore, for each $n$ one can find $x_n \in B_n \sms B_{n-1}$ such that $\vgf(x_n) \geq Ca^{\frac{n}{p}}$ with another constant $C$ independent of $n$. Since $\vgf$ is radially symmetric, by employing Harnack inequality, we derive that $\vgf$ has exponential growth. In particular, $\inf_{\BBR^N}\vgf>0$. By choosing $\psi=\vgf$ in the definition of $ \gl''(\CK_V,\BBR^N)$ in Definition \ref{gpe}, we derive that $\gl''(\CK_V,\BBR^N) \geq 0$. Thus  $\gl(\CK_V,\BBR^N) \leq \gl''(\CK_V,\BBR^N)$. This completes the proof. \qed \medskip

\noindent \textbf{Proof of Theorem \ref{equivalence}.} Statements (i)-(iii) follow directly from Proposition \ref{equal1} and Proposition \eqref{equal2}. \qed

\section{Qualitative properties of the principal eigenvalue}

\noindent \textbf{Proof of Theorem \ref{lim1}.}

\noindent (i)\; The fact that  $\lambda(\CL_\alpha,\Gw)$ is concave and nondecreasing is directly deduced by its variational characterization
$$\lambda(\CL_\alpha,\Gw)=\inf_{\phi \in C_c^{1}(\Gw) \sms\{0\}}\frac{\int_{\Gw}(\alpha|\nabla \gf|^p + V|\gf|^p)dx}{\int_{\O}|\phi|^pdx}.$$
We see that $$\mathcal{L}_\alpha[\phi]=\ga\left(-\Delta_p\phi+\frac{1}{\alpha}V(x)\phi^{p-1}\right) = \ga \CK_{\frac{1}{\ga}V}[\gf],$$
it follows that 
$$\gl(\CL_\ga,\BBR^N)=\ga \gl(\CK_{\frac{1}{\ga}V},\BBR^N).$$
By Theorem \ref{lim2} (ii), one has
$$\lim_{\alpha\to0}\ga \gl(\CK_{\frac{1}{\ga}V},\BBR^N)=\inf_{\O}V.$$
Therefore $\lim_{\ga \to 0}\gl(\CL_\ga,\BBR^N)=\inf_{\O}V$. \medskip

\noindent (ii)\; Let $\phi_\alpha>0$  be the eigenfunction associated with $\lambda(\CL_\ga,\Gw)$ with normalization $\|\phi_\alpha\|_{L^p(\O)}=1$. We have 
\bel{eigena} \left\{ \BAL \CL_\alpha [\phi_\alpha] &= \lambda(\CL_\ga,\Gw)\phi_\alpha^{p-1} \qquad &&\text{in } \Gw \\
\phi_\alpha &= 0 &&\text{on } \prt \Gw. 
\EAL \right. \ee
By the variational characterization of $\lambda(\CL_\ga,\Gw)$, we have
$$\lambda(\CL_\ga,\Gw)=\alpha\int_{\O}|\nabla\phi_\alpha|^pdx +\int_{\O}V\phi_\alpha^pdx \geq \alpha\int_{\O}|\nabla\phi_\alpha|^pdx+\inf_{\O}V. $$
By Poincar\'e  inequality, there exists a constant $C=C(N,p,\O)>0$ such that $$\int_{\O}|\nabla\phi_\alpha|^pdx \geq C \int_{\O}|\phi_\alpha|^pdx.$$ 
Combining the above estimates yields
$$\lambda(\CL_\ga,\Gw)\geq C\alpha+\inf_{\O}V, $$
which implies statement (ii). \medskip

\noindent (iii)\; We first prove \eqref{limsup}. It suffices to show that $\lambda(\CL_\alpha,\BBR^N) \leq \eta$ for any $\eta$ satisfying
\begin{equation}
\eta >\limsup_{|x|\to\infty}V(x).\label{11.1.2}
\end{equation}
Take $\gl$ satisfying \eqref{11.1.2}, there exists $R$ large enough such that
$$\inf_{|x|\geq R}(\eta-V(x))>0,$$
One sees that
$$\lambda(-\alpha\Delta_p -(\eta-V),B_R^c(0))<0.$$
Indeed, by the variational formula 
$$ \BAL
\lambda(-\alpha\Delta_p -(\eta-V(x)),B_R^c(0))&=\inf_{\su{\phi\in C_c^{1}(B_R^c(0))}{\phi\neq 0}}\frac{\alpha\int_{B_R^c(0)}|\nabla\phi|^pdx -\int_{B_R^c(0)}(\eta-V(x))|\phi|^pdx}{\int_{B_R^c(0)}|\phi|^pdx}\nonumber\\
&\leq\inf_{\su{\phi\in C_c^{1}(B_R^c(0))}{\phi\neq 0}}\frac{\alpha\int_{B_R^c(0)}|\nabla\phi|^pdx}{\int_{B_R^c(0)}|\phi|^pdx}-\inf_{B_R^c(0)}(\eta-V(x))\nonumber\\
&=-\inf_{B_R^c(0)}(\eta-V(x))<0,\nonumber
\EAL $$
where we use the fact that 
\begin{equation}
\inf_{\su{\phi\in C_c^{1}(B_R^c(0))}{\phi\neq 0}}\frac{\int_{B_R^c(0)}|\nabla\phi|^pdx}{\int_{B_R^c(0)}|\phi|^pdx}=0.\label{16.3.1}
\end{equation}
Let us prove (\ref{16.3.1}). For $r>R$, it is easily seen that 
\begin{equation}
\inf_{\su{\phi\in C_c^{1}(B_R^c(0))}{\phi\neq 0}}\frac{\int_{B_R^c(0)}|\nabla\phi|^pdx}{\int_{B_R^c(0)}|\phi|^pdx}=\lim_{r\to\infty}\lambda_{r},\label{16.3.2}
\end{equation}
where
$$\lambda_{r}=\inf_{\su{\phi\in C_c^{1}(B_{r}(0)\setminus B_R(0))}{\phi\neq 0}}\frac{\int_{B_{r}(0)\setminus B_R(0)}|\nabla\phi|^pdx}{\int_{B_{r}(0)\setminus B_R(0)}|\phi|^pdx}.$$
Take $x_0\in B_{r}(0)\setminus B_R(0)$ such that $|x_0|=\frac{r+R}{2}$. It is obvious that $$B_{\frac{r-R}{4}}(x_0)\subset B_{r}(0)\setminus B_R(0),$$
and hence
\begin{equation}
0\leq\lambda_{r}\leq \inf_{\su{\phi\in C_c^{1}(B_{\frac{r-R}{4}}(x_0))}{\phi\neq 0}}\frac{\int_{B_{\frac{r-R}{4}}(x_0))}|\nabla\phi|^pdx}{\int_{B_{\frac{r-R}{4}}(x_0))}|\phi|^pdx}.\label{16.3.3}
\end{equation}
By \cite[Theorem 4.3]{Mao}, 
$$\inf_{\su{\phi\in C_c^{1}(B_{\frac{r-R}{4}}(x_0))}{\phi\neq 0}}\frac{\int_{B_{\frac{r-R}{4}}(x_0))}|\nabla\phi|^pdx}{\int_{B_{\frac{r-R}{4}}(x_0))}|\phi|^pdx}\leq C(N,p,r,R),$$
where $C(N,p,r,R)$ is independent of $x_0$ and
\begin{equation}
C(N,p,r,R):=\left\lbrace\BAL
&\frac{p^{\frac{4^pp^2-p}{p-N}}}{n^{\frac{np-p}{p-N}}(p-1)^{p-1}(r-R)^p}\quad &&\text{if } N\neq p\\
&\frac{4^N N^Ne^{N-1}}{(N-1)^{N-1}(r-R)^N} &&\text{if } N=p.
\EAL\right.
\end{equation}
By letting $r\to\infty$ in (\ref{16.3.3}), one gets
$$\lim_{r\to\infty}\lambda_{r}=0.$$
By Theorem \ref{main0} (ii), we have
$$\lambda(-\alpha\Delta_p +(V-\eta),\R^N)\leq\lambda(-\alpha\Delta_p -(\eta-V),B_R^c(x_0))<0.$$
This implies that $\lambda(\CL_\alpha,\BBR^N)< \eta$
and hence we get \eqref{limsup}. \medskip

We next prove \eqref{liminf}. Put 
$$ \BAL \psi(r)&:=(e^{\gb r +1} + e^{-\gb r-1})^{-\gg} \quad r \geq 0, \\
\tl \psi(x)&:= \psi(|x|) \quad x \in \BBR^N,
\EAL $$
where $\gb>0$ and $\gg>0$  will be made precise later. Then we have 
\bel{formulation} \Gd_p \tl \psi(x) = (|\psi_r|^{p-2}\psi_r)_r + \frac{N-1}{r}|\psi_r|^{p-2}\psi_r \quad r=|x|>0 \ee
where $\psi_r$ denotes the first derivative of $\psi$. We next compute the left hand-side in \eqref{formulation}. It is easy to see that
$$ \BAL
\psi_r(r)&=-\gb\gg(e^{\gb r+1}+e^{-\gb r-1})^{-\gg-1}(e^{\gb r+1}-e^{-\gb r-1}), \\
|\psi_r|^{p-2}\psi_r &= - \gb^{p-1}\gg^{p-1}(e^{\gb r+1}+e^{-\gb r-1})^{-(\gg+1)(p-1)}(e^{\gb r+1}-e^{-\gb r-1})^{p-1}, \\
(|\psi_r|^{p-2}\psi_r)_r &= (p-1)\gb^p \gg^{p-1} g(r)^{p-2}[(\gg+1)g(r)^2-1]\psi^{p-1}
\EAL $$ 
where 
$$ g(r)=\frac{e^{2(\gb r+1)}-1}{e^{2 (\gb r+1)} +1}. $$
Since $\psi_r(r)<0$ for every $r>0$, it follows that $\Gd_p \tl \psi \leq (|\psi_r|^{p-2}\psi_r)_r$, namely
$$ \Gd_p \tl \psi \leq (p-1)\gb^p \gg^{p-1} g(r)^{p-2}[(\gg+1)g(r)^2-1]\psi^{p-1}, \quad r=|x|. $$
For $\gb>0$, since $g$ is increasing with respect to  $r \in (0,\infty)$, it follows that
$$ \frac{e^2-1}{e^2+1}=g(0) < g(r) < \lim_{r \to \infty}g(r)=1. $$
Therefore
\bel{form2} \Gd_p \tl \psi \leq c_p(p-1)\gb^p \gg^{p-1}[(\gg+1)g(r)^2-1]\psi^{p-1}, \quad c_p= \Big( \frac{e^2-1}{e^2+1} \Big)^{p-2}+1. \ee
Take arbitrarily $\vge>0$. Then there exists $R>0$ large enough such that 
$$ V(x) \geq \liminf_{|x| \to \infty}V - \vge \quad  \forall x \in B_R^c.$$
It follows that
$$ -\ga \Gd_p \tl \psi + V \tl \psi^{p-1} \geq \Big[-c_p(p-1)\ga\gb^p \gg^{p}  + (\liminf_{|x| \to \infty}V-\vge)\Big] \tl \psi^{p-1} \quad \text{in } B_R^c.
$$
Therefore, if we can choose $\ga$, $\gb$ and $\gg$ such that
\bel{cd1} \ga\gb^p \gg^{p}  \to 0 \quad \text{as } \ga \to \infty,  \ee
then there exists $\ga_1$ large enough such that for any $\ga \geq \ga_1$,
\bel{BRc} -\ga \Gd_p \tl \psi + V \tl \psi^{p-1} \geq (\liminf_{|x| \to \infty}V-2\vge) \tl \psi^{p-1} \quad \text{in } B_R^c.
\ee

We next choose $\gb>0$ small enough such that $g(r)^2<\frac{2}{3}$ for every $r \in (0,R)$. Then by \eqref{form2},
$$ \Gd_p \tl \psi(x) \leq  \frac{1}{3}c_p(p-1)\gb^p \gg^{p-1} (2\gg-1)\psi(|x|)^{p-1}, \quad x \in B_R.$$ 
This implies
$$ -\ga\Gd_p \tl \psi + V \tl \psi^{p-1} \geq \Big[ -\frac{1}{3}c_p(p-1)\ga \gb^p \gg^{p-1} (2\gg-1) + V \Big]\tl \psi^{p-1} \quad \text{in } B_R. $$
If we can choose $\ga$, $\gb$ and $\gg$ such that
\bel{cd2} \ga \gb^p \gg^{p-1} \to +\infty \quad \text{as } \ga \to \infty, \ee
then there exists $\ga_2>0$ large enough such that for any $\ga \geq \ga_2$,
\bel{BR} -\ga \Gd_p \tl \psi + V \tl \psi^{p-1} \geq (\liminf_{|x|\to \infty}V(x) - 2\vge)\tl \psi^{p-1} \quad \text{in } B_R. \ee  
By combining \eqref{BRc} and \eqref{BR} we deduce that 
\bel{RN} -\ga \Gd_p \tl \psi + V \tl \psi^{p-1} \geq (\liminf_{|x|\to \infty}V(x) - 2\vge)\tl \psi^{p-1} \quad \text{in } \BBR^N \ee  
provided that $\ga \geq \max\{\ga_1,\ga_2\}$, $\gb$ and $\gg$ small and \eqref{cd1} and \eqref{cd2} hold. We will choose $\gb$ and $\gg$ in the form
\bel{form} \gb=\ga^{-s_1} \quad \text{and} \quad \gg=\ga^{-s_2}, \quad s_1>0,s_2>0. \ee
Then \eqref{cd1} and \eqref{cd2} become
\bel{cds1} \BAL \ga^{1-ps_1 -ps_2} &\to 0 \quad \text{as } \ga \to \infty, \\
\ga^{1-ps_1 -(p-1)s_2} &\to \infty \quad \text{as } \ga \to \infty. \EAL \ee
We now choose $s_1=\frac{1}{2p}$ and $s_2=\frac{1}{2p-1}$ then \eqref{cds1}  holds and thus we get \eqref{RN}. This implies
$$ \lambda(\CL_\alpha,\BBR^N) \geq  (\liminf_{|x| \to \infty}V - 2\vge). $$
Since $\vge>0$ is arbitrary, we derive
$$ \liminf_{\ga \to \infty}\lambda(\CL_\alpha,\BBR^N) \geq \liminf_{|x| \to \infty}V. $$ \qed

Before proving Theorem \ref{lim2}, we first need the following result.

\begin{proposition} \label{proplim} Assume $\{V_\ga\}$ is a sequence of functions in $L_{loc}^\infty(\Gw)$. For $\ga>0$, denote
\bel{Ka} \CK_\ga[\phi]:=-\Delta_p\phi+V_\ga\,\phi^{p-1}, \quad \gf \geq 0. \ee
Assume $V \in L_{loc}^\infty(\Gw)$  and $V_n\rightharpoonup V$ in $L^1_{loc}(\Gw)$ as $\ga \to \infty$. Then 
$$\limsup_{\ga \to\infty}\lambda(\CK_\ga,\Gw)\leq \lambda(\CK_V,\Gw).$$
\end{proposition}
\Proof Put $\ol\lambda:=\limsup_{\ga\to\infty}\gl(\CK_\ga,\Gw)$. Obviously, $\ol\lambda\in(-\infty,\infty)$. Therefore there exists a subsequence, still denoted by the same notation, such that $\gl(\CK_\ga,\Gw) \to \ol\lambda$ as $\ga \to \infty$. We denote by $\phi_\ga$  the generalized principal eigenfunction associated to $\gl(\CK_\ga,\Gw)$ with normalization $\phi_\ga(0)=1$.  Since $\{V_\ga\}$ is locally uniformly bounded in $\Gw$, by the Harnack inequality we deduce that $\{\gf_\ga\}$ is locally uniformly bounded in $C^1_{loc}(\Gw)$. By local regularity results  for quasilinear elliptic equations (see \cite{Di}) and a standard argument, we deduce that, up to a subsequence, $\{\phi_\ga\}$ converges in $C^1_{loc}(\O)$ to  a nonnegative function $\ol \phi$ which is a weak solution of 
$$ - \Gd _p\bar \gf + V \bar \gf^{p-1} = \bar  \gl \,\bar \gf^{p-1} 	\quad \text{in } \Gw $$
and satisfies $\bar \phi(0)=1$. By the strong maximum principle, one has $\ol\phi>0$ in $\O$. Therefore, by choosing $\psi=\bar \gf$ in the definition of $ \lambda(\CK_V,\O)$ in \eqref{lambda}, we deduce that $ \lambda(\CK_V,\O)\geq \ol\lambda$. \qed \medskip

We are led to

\noindent \textbf{Proof of Theorem \ref{lim2}.}

\noindent (i)\, Observe that $\{V_\ga\}$ is locally uniformaly bounded in $\BBR^N$ and $V_\alpha \to V(0)$ in $L_{loc}^1(\R^N)$ as $\alpha \to 0$. Define
$$ \CA [\gf]:= -\Delta_p \gf + V(0) \gf^{p-1} $$
and denote by $\lambda(\CA,\R^N)$ the generalized principal eigenvalue of $\CA$ in $\R^N$. 
By  Proposition \ref{proplim}, we get
\begin{equation} \label{limsup1} \limsup_{\alpha \to 0}\lambda(\CK_\ga,\BBR^N) \leq \lambda(\CA,\R^N). 
\end{equation}
On the other hand, by Theorem \ref{main0} (i), for any $\alpha>0$,
\begin{equation} \label{limsup2} \lambda(\CK_\ga,\BBR^N) \geq \inf_{\R^N}V_\alpha = V(0). 
\end{equation}
Hence from Theorem \ref{equivalence}, \eqref{limsup1} and \eqref{limsup2} we get
\begin{equation} \label{limsup3} V(0) \leq \liminf_{\alpha \to 0} \lambda(\CK_\ga,\BBR^N) \leq \limsup_{\alpha \to 0} \lambda(\CK_\ga,\BBR^N) \leq \lambda(\CA,\R^N)=\lambda'(\CA,\R^N). 
\end{equation}

We next prove that 
\begin{equation} \label{limsup4} \lambda'(\CA,\R^N) \leq V(0).
\end{equation}
To this end, from the Definition \ref{gpe}, it is sufficient to show that there exists $0<\psi \in W^{1,p}(\R^N)$ such that $\CA [\psi] \leq V(0)\psi^{p-1}$ in the weak sense in $\R^N$. Define
\bel{psi} \psi(x):=\chi_{B_1}(x) + e^{\frac{N-1}{p-1}(1-|x|)}\chi_{B_1^c}(x) \ee
then $\psi$ is the desired function. Indeed, we see that  $\Delta_p \psi =0$ in $B_1$ and 
$$ \Delta_p \psi(x) = \frac{(N-1)^p}{(p-1)^{p-1}}e^{(N-1)(|x|-1)}\Big( 1 - \frac{1}{|x|}  \Big) \geq 0 \quad \text{in } B_1^c.
$$
It follows that $\CA [\psi] \leq V(0)\psi^{p-1}$ in the weak sense in $\R^N$. By combining \eqref{limsup3} and \eqref{limsup4} we obtain \eqref{lubu}.

Next, let us  prove \eqref{lubu1}. By a similar argument as above and by Proposition \ref{proplim}, we obtain
$$\limsup_{\alpha\to\infty}\gl(\CK_\ga,\BBR^N)\leq \lambda(\mathcal{B},\R^N),$$
where $$\mathcal{B}[\gf]:= -\Delta_p \gf +\ol V(x) \gf^{p-1}\quad\quad\textrm{and}\quad\quad \ol V(x):=\left\{ \BAL
&\sup_{\R^N}V(x)\quad &&\text{if } x\neq 0\\
&\inf_{\R^N}V(x)&&\text{if } x= 0.
\EAL \right.$$
Thanks to variational characterization of $\lambda(\mathcal{B},\R^N\setminus\{0\})$, one has
\begin{eqnarray}
\lambda(\mathcal{B},\R^N\setminus\{0\})&=&\inf_{\su{\phi \in C^1_c(\R^N\setminus\{0\}) \sms \{0\}}{\|\phi\|_{L^p(\R^N)}=1}}\int_{\R^N\setminus\{0\}}(|\nabla \gf|^p + \ol V|\gf|^p)dx\nonumber\\
&=&\inf_{\su{\phi \in C^1_c(\R^N\setminus\{0\}) \sms \{0\}}{\|\phi\|_{L^p(\R^N)}=1}}\int_{\R^N\setminus\{0\}}(|\nabla \gf|^p + \sup_{\R^N}V|\gf|^p)dx\nonumber\\
&=&\inf_{\su{\phi \in C^1_c(\R^N) \sms \{0\}}{\|\phi\|_{L^p(\R^N)}=1}}\int_{\R^N}(|\nabla \gf|^p + \sup_{\R^N}V|\gf|^p)dx\nonumber\\
&=&\sup_{\R^N}V.\nonumber
\end{eqnarray}
Hence 
$$\limsup_{\alpha\to\infty}\gl(\CK_\ga,\BBR^N)\leq \lambda(\mathcal{B},\R^N)\leq \lambda(\mathcal{B},\R^N\setminus\{0\})= \sup_{\R^N}V.$$
Let us show that 
$$\liminf_{\alpha\to\infty}\lambda_\alpha\geq \sup_{\R^N} V.$$
By the variational characterization of $\lambda_\alpha$, we have
\begin{eqnarray}
\gl(\CK_\ga,\BBR^N)&=&\inf_{\su{\phi \in C^1_c(\R^N) \sms \{0\}}{\|\phi\|_{L^p(\R^N)}=1}}\int_{\R^N}(|\nabla \gf|^p +  V(\alpha x)|\gf|^p)dx\nonumber\\
&=&\inf_{\su{\phi_\alpha \in C^1_c(\R^N) \sms \{0\}}{\|\phi_\alpha\|_{L^p(\R^N)}=1}}\int_{\R^N}(\alpha^p|\nabla \gf_\alpha|^p +V(x)|\gf_\alpha|^p)dx\nonumber\\
&=&\lambda(-\alpha^p\Delta_p +V,\R^N),\nonumber
\end{eqnarray}
where $\phi_\alpha(x)=\phi(x/\alpha)$. Applying (\ref{liminf}), we derive
$$\liminf_{\alpha\to\infty}\gl(\CK_\ga,\BBR^N)\geq \liminf_{|x|\to\infty}V(x)= \sup_{\R^N}V(x).$$
This concludes the proof.
\medskip

\noindent (ii)\, By Theorem \ref{main0}, 2i), $\lambda(\CK_\ga,\BBR^N)\geq \alpha\inf_{\O} V$. Therefore,  to prove the statement (ii), it is sufficient to show that 
\begin{equation}
\limsup_{\alpha\to\infty}\frac{\lambda(\CK_\ga,\BBR^N)}{\alpha}\leq \inf_{\Gw} V.\label{10.1.1}
\end{equation}
Since $V$ is upper semi-continuous, for any $\epsilon>0$, there exists a ball $B\subset\O$ such that 
$$ V(x)<\inf_{\O} V+\epsilon \quad \forall x \in B.$$ 
Let $\lambda_B$ and $\phi_B$ be the Dirichlet principal
eigenvalue and a corresponding eigenfunction of the operator of $-\Delta_p$ in $B$. For $\alpha>\lambda_B/\epsilon$, we have
$$ \BAL
\CK_\alpha[\gf_B] -\alpha(\inf_\O V+2\epsilon)\phi_B^{p-1}&=(\lambda_B+\alpha(V(x)-\inf_\O V-2\epsilon))\phi_B^{p-1}\\
&<(\lambda_B-\alpha\epsilon)\phi_B^{p-1}<0\nonumber
\EAL $$
By taking $\psi=\phi_B$ in the definition of $\lambda'(\CK_\ga,B)$ in Definition \ref{gpe}, we  get 
$$\lambda'(\CK_\alpha,B)\leq \alpha(\inf_\O V+2\epsilon).$$ 
Since $\lambda'(\CK_\alpha,B)=\lambda(\CK_\alpha,B)\geq \gl(\CK_\ga,\BBR^N)$, it follows that
$$\frac{\gl(\CK_\ga,\BBR^N)}{\alpha}\leq \inf_\O V+2\epsilon. $$
We achieve (\ref{10.1.1}) due to the arbitrariness of $\epsilon$. \medskip

\noindent (iii)\; This can be proved by using a similar argument as above and thus we omit the detail.\qed
%%%%%%%%%%%%%%%%%%%%%%%%%%%%%%%%%%%%%%%%%%%%%%%%%%%%
%%%%%%%%%%%%%%%%%%%%%%%%%%%%%%%%%%%%%%%%%%%%%%%%%%%%
%%%%%%%%%%%%%%%%%%%%%%%%%%%%%%%%%%%%%%%%%%%%%%%%%%%%
%%%%%%%%%%%%%%%%%%%%%%%%%%%%%%%%%%%%%%%%%%%%%%%%%%%%
%%%%%%%%%%%%%%%%%%%%%%%%%%%%%%%%%%%%%%%%%%%%%%%%%%%%
%%%%%%%%%%%%%%%%%%%%%%%%%%%%%%%%%%%%%%%%%%%%%%%%%%%%
\section{Simplicity of the generalized principal eigenvalue and maximum principle}

\noindent \textbf{Proof of Theorem \ref{simplicity1}.} Let $v \in W_{loc}^{1,p}(\BBR^N)$ be a positive weak solution of \eqref{eigen} in $\BBR^N$. 

\noindent \textit{Step 1: We show that} 
$$ {\mathbb D}:=\{ d>0: u \leq d v \quad \text{in } \BBR^N \} \ne \emptyset. $$
Indeed, suppose by contradiction that  ${\mathbb D} = \emptyset$. Then for each $n \in \BBN$, there exists $x_n$ such that $u(x_n) \geq n v(x_n)>0$. We consider two cases: (i) up to a subsequence, $x_n \to x^* \in \BBR^N$, (ii) up to a subsequence, $|x_n| \to \infty$. In case (i), by passing to the limit, $u(x^*)=\infty$, which is a contradiction. In case (ii),  since $u$ is a solution of \eqref{eigen} of minimal growth at infinity and $v$ is a solution of \eqref{eigen} then there exists $R^*>0$ and $k>0$ such that $u \leq k v$ in $B_{R^*}^c$. Then we choose $n_0>2k$ large enough such that $|x_n|>R^*$ for every $n \geq n_0$. Therefore
$$ u(x_n) \geq n v(x_n) > 2k v(x_n) \geq 2u(x_n) \quad \forall n \geq n_0, $$
which is a contradiction. \smallskip

\noindent \textit{Step 2: Scaling process.}
Put $ \ell:=\inf \BBD$ then $u \leq \ell v$ in $\BBR^N$. We consider two cases.

\textit{Case 1:} There exists $\tilde x \in \BBR^N$ such that $u(\tilde x)=\ell v(\tilde x)$. Put $w:=\ell v - u$ then $w \geq 0$ in $\BBR^N$ and $w(\tl x)=0$. Put $$W:=V-\lambda(\CK_V,\R^N).$$ 
We have
\bel{trans} \BAL  0 &= -\Delta_p (\ell v)+W (\ell v)^{p-1} - (-\Delta_p u+Wu^{p-1}) \\
&= -(\Delta_p (\ell v) - \Delta_p u) +  W [(\ell v)^{p-1} - u^{p-1}] \\
&= -\sum_{i,j}\partial_{i}(a_{ij}(x)\partial_j w) + b(x)w
\EAL \ee
where 
$$ \BAL a_{ij}:&=|t_i \nabla (\ell v) + (1-t_i)\nabla u|^{p-4}\Big[ \gd_{ij}|t_i \nabla (\ell v) + (1-t_i)\nabla u|^{2} \\
& + (p-2)(t_i \partial_i (\ell v) + (1-t_i)\partial_i u)(t_i \partial_j (\ell v) + (1-t_i)\partial_j u)\Big]
\EAL $$
with some $t_i \in (0,1)$ and 
$$ b(x):=\left\{ \BAL &W(x)\frac{(\ell v(x))^{p-1}-u(x)^{p-1}}{\ell v(x)-u(x)} \quad &&\text{if } \ell v(x) \neq u(x)\\
&0 &&\text{if } \ell v(x) = u(x).
\EAL \right.
$$
We see that $\nabla (\ell v(\tl x))=\nabla u(\tl x) \ne 0$, hence
$$ a_{ij}(\tl x)=|\nabla u(\tl x)|^{p-4}\Big[ \gd_{ij}|\nabla u(\tl x)|^2 + (p-2)\prt_i u (\tl x) \prt_j u(\tl x) \Big].     
$$
Therefore the matrix $(a_{ij}(\tl x))$ is positive definite. Consequently, $(a_{ij})$ is also positive definite in ball $B_\gd(\tl x)$ for some small $\gd>0$. We see that $b$ is bounded in $B_\gd(\tl x)$. From \eqref{trans}, it follows that
$$ -\sum_{i,j}\partial_{i}(a_{ij}\partial_j w) + b^+ w \geq 0 \quad \text{in } B_{\gd}(\tl x).$$
By the strong maximum principle for linear equations with principal part of divergence form \cite[Theorem 8.19]{GTbook}, we deduce that $w=0$ in $B_\gd(\tl x)$. In light of the strong comparison principle \cite[Theorem 3.2]{FrPi}, $w=0$ in $\BBR^N$, hence $u=\ell v$ in $\BBR^N$. 

\textit{Case 2: $u(x)<\ell v(x)$ for every $x \in \BBR^N$ and there exists a sequence $\{x_n\}$ such that $|x_n| \to \infty$ as $n \to \infty$} and 
\bel{liml} \lim_{n \to \infty}\frac{u(x_n)}{v(x_n)}=\ell. \ee
Put $r_n:=|x_n|$ and $M_n:=\sup_{\prt B_{r_n}(0)}v(x)$. Set $u_n(x):=M_n^{-1}u(r_n x)$, $v_n(x):=M_n^{-1}v(r_n x)$ and $W_n(x)=r_n^{p-1} W(r_n x)$. Then $u_n$ and $v_n$ are solutions of 
$$ - \Gd_p u + W_n u^{p-1} = 0 \quad \text{in } \BBR^N. $$
By \eqref{condV}, there exists $C>0$ and $R_0>0$ such that
$$ W(x) \leq C |x|^{-(p-1)} \quad \forall x \in B_{R_0}^c. $$
This implies that, for $n$ large enough (such that $r_n >2R_0$),
$$ |W_n(x)|=r_n^{p-1}|W(r_n x)| \leq Cr_n^{p-1}|r_n x|^{-(p-1)}=C|x|^{-(p-1)} \leq C' \quad \text{in } B_{\frac{1}{2}}^c.$$
Since $\sup_{\prt B_1}v_n=1$, by Harnack inequality, for any $\rho>1$ there exists a positive constant $C_{\rho}$  independent of $n$ such that $v_n \leq C_\rho$ in $B_{\rho} \sms B_{\frac{3}{4}}$ for every $n$. It follows that $u_n \leq C_\rho\ell$ in $B_{\rho} \sms B_{\frac{3}{4}}$ for every $n$. Note that, up to a subsequence, $\{W_n\}$ converges to a function $\tl W \in L^\infty( B_{\frac{3}{4}}^c)$ in weak-star topology of $L^\infty$. Therefore, by regularity results for quasilinear elliptic equations and a standard argument, we deduce that, up to a subsequence, $\{u_n\}$ and $\{v_n\}$ converge in the $C_{loc}^1(B_{\frac{3}{4}}^c)$ topology to functions $\tl u$ and $\tl v$ respectively which are weak solutions of equation
$$ - \Gd_p u + \tl W u^{p-1} = 0 \quad \text{in }  B_{\frac{3}{4}}^c. $$
Put $y_n:=r_n^{-1}x_n$ then $|y_n|=1$ for every $n$. Hence, up to a subsequence, $y_n \to \tl y \in \prt B_1$ and $\gk_n \to \tl \gk \leq \ell$. Therefore if $\nabla \tl u(\tl y) \ne 0$ then we can use the strong comparison principle as in Case 1 to deduce that $\tl u=\ell \tl v$ in $B_{\gd}(\tl y)$ for some $\gd>0$ small. Consequently, due to the strong comparison principle \cite[Theorem 3.2]{FrPi}, $\tl u = \ell \tl v$ in $B_{\frac{3}{4}}^c$.

Now we prove that $\nabla \tl u(\tl y) \ne 0$. Indeed, from \eqref{liml}, we have $\tl u(\tl y)=\ell \tl v(\tl y)$.
Since $\{v_n\}$ converges to $\tl v$ in $C_{loc}^1(\BBR^N)$. Therefore, there exists $n_1$ large enough such that for every $n \geq n_1$, 
\bel{al1} |\nabla \tl v(x)|>\frac{1}{2}|\nabla v_{n}(x)| \quad \forall x \in B_2(0). \ee 
From \eqref{nablauu}, we deduce that there exist $\vge>0$ and $n_2>0$ such that for every $n \geq n_2$,
\bel{al2}  \frac{r_n |\nabla v(r_n \tl y)|}{v(r_n \tl y)}> \vge. \ee
Next fix $n_0>\max\{ n_1,n_2 \}$. By applying Harnack inequality for $v_{n_0}$, we derive that there exists a positive constant $C=C(N,p,\norm{V}_{L^\infty(\BBR^N)},r_{n_0})$ such that
\bel{al3} M_{n_0} \leq C \inf_{\prt B_{r_{n_0}}(0)}v_{n_0} \leq C v(r_{n_0}\tl y). \ee
Combining \eqref{al2} and \eqref{al3} yields 
\bel{al4} |\nabla v_{n_0}(\tl y)|=\frac{r_{n_0} |\nabla v(r_{n_0} \tl y)|}{M_{n_0}}=\frac{r_{n_0} |\nabla v(r_{n_0} \tl y)|}{v(r_{n_0}\tl y)}\frac{v(r_{n_0}\tl y)}{M_{n_0}}>\frac{\vge}{C}>0.
\ee
From \eqref{al1} and \eqref{al4}, we get $ |\nabla \tl v(x)|>\frac{\vge}{2C}>0$. Therefore 
$$|\nabla \tl u(\tl y)|=\ell |\nabla \tl v(\tl y)|=\frac{\ell \vge}{C}>0.$$
\noindent \textit{Step 3: End of proof.} Put
$$\gk_n:=\inf_{x \in \prt B_{r_n}(0)}\frac{u(x)}{v(x)}$$
then $\gk_n \leq \ell$. Therefore, up to a subsequence, $\gk_n \to \gk \leq \ell$. From Step 2, we deduce that $\gk=\ell$. Consequently, for every $\ge>0$ there exists $n_\ge$ such that for any $n \geq n_\ge$,
$$   (\ell-\ge)v(x) \leq  u(x) \leq \ell v(x) \quad \forall x \in \prt B_{r_n}(0).
$$
By the weak comparison principle \cite[Theorem 3.1]{FrPi}, $(\ell -\ge)v \leq u$ in $B_{r_n}(0)$. Letting $n \to \infty$ and $\ge \to 0$ implies $\ell v \leq u$ in $\BBR^N$. Thus $u=\ell v$ in $\BBR^N$. \qed

\begin{proposition} \label{class1}
	Assume $p\geq 2$ and $V \in L^\infty(\BBR^N)$ such that
	\bel{kcond} \lambda(\CK_V,\BBR^N) < \liminf_{|x| \to \infty}V(x) - \mu
	\ee
	for some $\mu>0$. Let $u$ be a positive weak solution of \eqref{eigen}. If
	\bel{c1}\limsup_{|x|\to\infty}\frac{u(x)}{e^{\tilde{\omega}|x|}}<\infty \quad \text{with} \quad
	\tilde{\omega}:=\left(\frac{\mu}{N(p-1)}\right)^{\frac{1}{p}}.\ee
	Then 
	\bel{c2'} \lim_{|x|\to\infty}e^{\tilde{\omega}|x|^{}}u(x)=0.\ee
\end{proposition}

\Proof Since $u$ satisfies \eqref{eigen}, by \eqref{kcond}, for any $\varepsilon>0$ there exists $R=R(\varepsilon)$ such that
$$-\Gd_p u + (\mu+\varepsilon)u^{p-1} \leq 0 \quad \text{in the weak sense in } B_R^c.$$
Let $\CL_\vge[\phi]:=-\Delta_p\phi+(\mu+\varepsilon)\phi^{p-1}$. It is easy to see that $\CL_\vge [u]\leq 0$ in the weak sense in $B_R^c$. For any $\rho>0$, set
$$ w^1_\rho(x):=e^{(R+\rho)^{}(\tau-\omega)}e^{\omega|x|^{}}, \qquad w^2_\rho(x):= e^{R^{}(\tau+\omega)}e^{-\omega|x|^{}},$$
$$ w_\rho:=w^1_\rho+w^2_\rho$$
where $\omega,\tau$ and $R$ will be chosen later. Observe that, in $\BBR^N \sms \{0\}$, 
\begin{equation}
\Delta_p w_\rho=(p-2)|\nabla w_\rho|^{p-4}\left\langle D^2w_\rho \, \nabla w_\rho,\nabla w_\rho\right\rangle+|\nabla w_\rho|^{p-2}\Delta w_\rho.
\end{equation}
By Cauchy-Schwarz inequality,  one has 
\begin{equation} \label{prop1.1}
\Delta_p w_\rho \leq((p-2)N\max_{ij}|\partial_{ij} w_\rho|+|\Delta w_\rho|)|\nabla w_\rho|^{p-2}.
\end{equation}
Next, we look for an upper bound for the right-hand side of \eqref{prop1.1}. Direct computation yields, for every $x \neq 0$, 
\begin{equation}\nonumber
\nabla w_\rho=\omega\frac{x}{|x|}\,
 w_\rho^1-\omega\frac{x}{|x|}\, w_\rho^2, 
\end{equation}
thus
\begin{equation}\label{NV1_expo1}
|\nabla w_\rho|^{p-2}\leq \omega^{p-2}w_\rho^{p-2}.
\end{equation}
For every $1 \leq i,j\leq N$ and $x \neq 0$,  
\begin{eqnarray}
\nonumber\partial_{ij}w_\rho&=&\omega^2\frac{x_ix_j}{|x|^{2}}w_\rho^1-\omega\frac{x_ix_j}{|x|^{3}}w_\rho^1+\delta(i-j)\omega\frac{w_\rho^1}{|x|^{}}\nonumber + \omega^2\frac{x_ix_j}{|x|^{2}}w_\rho^2 \omega\frac{x_ix_j}{|x|^{3}}w_\rho^2-\delta(i-j)\omega\frac{w_\rho^2}{|x|^{}} 
\end{eqnarray}
where $\delta$ is the Dirac function. Since $|x_ix_j|\leq |x|^2$, it follows that 
\bel{esti1} \BAL
|\partial_{ij}w_\rho|&\leq  \omega^2 w_\rho +\omega\frac{w_\rho}{|x|}+\omega\frac{w_\rho}{|x|}= (\gw^2|x|^{}+2\omega)\frac{w_\rho}{|x|}
\EAL \ee
and hence
\bel{esti2}
|\Delta w_\rho|\leq N(\gw^2|x|^{}+2\omega)\frac{w_\rho}{|x|}.
\ee
Combining \eqref{prop1.1}-\eqref{esti2}, we have 
$$  \Delta_pw_\rho
\leq N(p-1)\omega^{p-1}(\gw|x|+2)\frac{w_\rho^{p-1}}{|x|}. \nonumber
$$
Put $A:=2N(p-1)\omega^{p-1}.$ As $|x|\geq R$, one gets
\bel{sup1}
\CL_\vge[w_\rho]\geq w^{p-1}_\rho\left[-N(p-1)\omega^p- \frac{A}{|x|}+\mu+\varepsilon\right].
\ee
One can choose $R$ and $\gw$ such that the right-hand side of \eqref{sup1} is nonnegative. Indeed, since $|x|^{-1}\to0$ as $|x|\to\infty$,  there exists $R(\vge)$ such that, for every $R>R(\vge)$,  $A|x|^{-1}\leq \varepsilon/2\forevery x \in B^c_{R}.$ Take 
$$ \omega:=\left(\frac{2\mu+\varepsilon}{2N(p-1)}\right)^{\frac{1}{p}}, $$
 we obtain $\CL_\ge[w_\rho] \geq 0$ in $B_{R+\rho} \sms B_\rho$. We next show that $w_\rho$ dominates $u$ on $\prt B_{R+\rho} \cup \prt B_\rho$. 
Indeed, by \eqref{c1}, one can finds $C>0$ such that $u(x)\leq Ce^{\tilde{\omega}|x|^{}}$ in $\R^N$. Therefore, we can take $\tau$ arbitrarily in $(\tilde{\omega},\omega)$ and $R$ sufficiently large such that for any $\rho>0$, one has
\begin{equation}
\left\lbrace\begin{array}{ll}
w_\rho(x)\geq e^{ R^{}\tau}\geq C e^{ R^{}\tilde{\omega}}\geq u(x),& \textrm{as $|x|=R$}\\
w_\rho(x)\geq e^{ (R+\rho)^{}\tau}\geq C e^{ (R+\rho)^{}\tilde{\omega}}\geq u(x),&\textrm{as $|x|=R+\rho$}.\nonumber
\end{array}\right.
\end{equation}
Fix such $\omega,\tau$ and $R$. Applying the weak comparison principle \cite{GS}, we obtain 
$$u(x)\leq w_\rho(x)=e^{(R+\rho)^{}(\tau-\omega)}e^{\omega|x|^{}}+e^{R^{}(\tau+\omega)}e^{-\omega|x|^{}}\quad\quad\textrm{in $B_{R+\rho}\setminus B_R$}.$$
Sending $\rho\to\infty$ yields
$$u(x)\leq e^{R^{}(\tau+\omega)}e^{-\omega|x|^{}}\quad\quad\textrm{in $\R^N\setminus B_R$}.$$
The fact $\omega>\tilde{\omega}$ confirms the proof.\qed \medskip

\begin{lemma} \label{exgrowth1} Assume $V \in L^\infty(\BBR^N)$. There exist positive constants $C$ and $\gb$ depending on $N,p,\| V \|_{L^\infty(\BBR^N)}$ such that if $u \in W_{loc}^{1,p}(\BBR^N)$ is a positive weak solution of \eqref{eqKV} in $\BBR^N$ then 
\bel{exgrowth} u(x) \leq C e^{\gb |x|}u(0) \quad \forall x \in \BBR^N. \ee 
\end{lemma}
\begin{proof} By Harnack inequality \cite{Se} (see Theorem 5 for $p<N$, Theorem 6 for $p=N$ and Theorem 9 for $p>N$), we deduce that there exists a positive constant $C$ depending on $N,p,\| V \|_{L^\infty(\BBR^N)}$ such that
\bel{Har1} u(x) \leq C u(0) \quad \forall x \in B_1(0). \ee 
Take $x_0 \in \prt B_1(0)$. We claim that for any nonnegative $m$, there holds
\bel{ith} u(x) \leq C^{m+1}u(0) \quad \forall x \in B_1(m x_0). \ee
We will prove \eqref{ith} by induction on $m$. Obviously, \eqref{ith} holds true for $m=0$. Suppose that \eqref{ith} is valid for some positive integer $m$,  we will show that \eqref{ith} also holds true for $m+1$. Indeed, observe that since $V \in L^\infty(\BBR^N)$, \eqref{Har1} still holds if we replace $u$ by $u(\cdot + y)$  for every $y \in \BBR^N$. In particular, with $|x_0|=1$, by replacing $u$ by $u(\cdot + (m+1)x_0)$, we get
$$ u(x+(m+1)x_0) \leq Cu((m+1)x_0) \quad \forall x \in B_1(0). $$
By changing the variable and \eqref{ith}, we get
$$ u(x) \leq Cu((m+1)x_0) \leq C^{m+2}u(0) \quad \forall x \in B_1((m+1)x_0).
$$
Thus we have proved \eqref{ith}.

By \eqref{ith} we deduce that 
$$ u(x) \leq C^{|x|+2}u(0) \quad \forall x \in B_1(m x_0). $$
This implies \eqref{exgrowth}.
\end{proof}

\noindent \textbf{Proof of Theorem \ref{simplicity2}.} 

\noindent \textit{Step 1: Exponential decay.} We prove that there exists $\mu>0$ such that if \eqref{decay} holds then every positive eigenfunction associated with  $\lambda(\CK_V,\R^N)$ decays exponentially. 
	
	Indeed, let $\gf \in W^{1,p}_{loc}(\R^N)$ be a positive eigenfunction associated with $\lambda(\CK_V,\R^N)$ then $\gf$ is a weak solution of \eqref{eigen}. By  Lemma \ref{exgrowth1}, there exist a constant $C_\gf>1$ depending on $\gf, N,p,\| V \|_{L^\infty(\BBR^N)}$ and a constant $\gb>0$ depending on $N,p,\| V \|_{L^\infty(\BBR^N)}$ such that
	$$\gf(x)\leq C_\gf\,e^{\beta|x|} \quad \forall x \in \BBR^N.$$
	Put $\mu:=N(p-1)\gb^p $. If \eqref{decay} holds then by invoking Proposition \ref{class1}, we deduce that 
	$$ \gf(x) \leq C'_\gf e^{-\gb |x|} \quad \forall x \in \BBR^N$$
	where $C'_\gf$ is a positive constant depending on $\gf,N,p$ and $\| V \|_{L^\infty(\BBR^N)}$.  \medskip
	
	\noindent \textit{Step 2: Simplicity.} Let $\gf,\vgf \in W^{1,p}_{loc}(\R^N)$ be two positive eigenfunctions associated with $\gl(\CK_V,\BBR^N)$. We will prove that, under condition \eqref{decay}, there exist a constant $k$ such that $\gf=k \vgf$ in $\R^N$. Indeed, if \eqref{decay} holds then by Step 1 $u$ and $v$ are decay exponentially.
	In light of the  regularity for quasilinear elliptic equations \cite{Lieberman} and the Harnack inequality, we deduce that $\gf, \vgf\in W^{1,p}(\R^N)$. Therefore, without lost of generality, we can normalize $u$ and $v$ so that $\norm{\gf}_{L^p(\BBR^N)}=\norm{\vgf}_{L^p(\BBR^N)}=1$. Set $$\vartheta:=\zeta^{1/p} \quad \text{with} \quad \zeta:=\frac{\gf^p+\vgf^p}{2},$$
	then $\| \vartheta \|_{L^p(\R^N)}=1$. Set 
	$$\theta:=\frac{\gf^p}{\gf^p+\vgf^p}\in(0,1).$$ 
	One has
	$$\nabla \vartheta=\zeta^{-1+\frac{1}{p}}\left(\frac{\phi^{p-1}\nabla\varphi+\varphi^{p-1}\nabla\phi}{2}\right).$$
	The convexity of the map $s\mapsto |s|^{p}$ yields
	\begin{eqnarray*}
		|\nabla \vartheta|^p&=&\zeta^{1-p}\left|\frac{\phi^{p-1}\nabla\phi+\varphi^{p-1}\nabla\varphi}{2}\right|^p \\
		&=&\frac{\zeta^{}}{2^p}\left|\theta(x)\frac{\nabla\phi}{\phi}+(1-\theta(x))\frac{\nabla\varphi}{\varphi}\right|^p\nonumber\\
		&\leq&\zeta \left(\theta(x)\left|\frac{\nabla\phi}{\phi}\right|^p+(1-\theta(x))\left|\frac{\nabla\varphi}{\varphi}\right|^p\right)  \nonumber \\ 
		&=&\frac{|\nabla u|^p+|\nabla v|^p}{2}.\nonumber
	\end{eqnarray*}
	The equality holds  if and only if $\frac{\nabla\phi}{\phi}=\frac{\nabla \varphi}{\varphi}$ in $\R^N$. Hence,
	$$\int_{\R^N}|\nabla \vartheta|^pdx\leq\frac{1}{2}\left(\int_{\R^N}|\nabla \phi|^pdx+\int_{\R^N}|\nabla \varphi|^pdx\right).$$
	It follows that
	$$
	\mathcal{F}(\vartheta)\leq \frac{1}{2}(\mathcal{F}(\phi)+\mathcal{F}(\varphi)),
	$$
	where 
	$$\mathcal{F}(w):=\int_{\R^N}(|\nabla w|^p+V|w|^p)dx\qquad w\in W^{1,p}(\R^N).$$
	Since both $\phi,\varphi\in W^{1,p}(\R^N)$ satisfy \eqref{eigen}, we deduce
	\begin{equation}\label{decay2}
	\mathcal{F}(\vartheta)\leq \frac{1}{2}(\mathcal{F}(\phi)+\mathcal{F}(\varphi))=\gl(\CK_V,\BBR^N).
	\end{equation}
	Let $\{\vartheta_n\} \sbs C_c^1(\BBR^N)$ be a sequence converging to $\vartheta$ in $W^{1,p}(\BBR^N)$. Since $\vartheta>0$, we infer that $\vartheta_n>0$ for $n$ large enough. For large $n$, by Theorem \ref{main0} iii), we have
	$$ \lambda(\CK_V,\BBR^N)\leq\frac{\int_{\BBR^N}(|\nabla \vartheta_n|^{p}+V \vartheta_n^p)dx}{\int_{\BBR^N}\vartheta_n^p\,dx}.
	$$
	Letting $n\to \infty$ yields
	\begin{equation} \label{decay3}
	\lambda(\CK_V,\BBR^N)\leq\frac{\int_{\BBR^N}(|\nabla \vartheta|^{p}+V \vartheta^p)dx}{\int_{\BBR^N}\vartheta^p\,dx} = \CF(\vartheta).
	\end{equation}
	By \eqref{decay2} and \eqref{decay3}, we get 
	$$ \CF(\vartheta)=\gl(\CK_V,\BBR^N). $$
	The above equality holds if and only if  $\frac{\nabla\phi}{\phi}=\frac{\nabla \varphi}{\varphi}$ in $\R^N$, which implies $\nabla \left(\frac{\phi}{\varphi}\right)=0$ in $\BBR^N$. Therefore there exists $k>0$ such that $\phi=k\varphi$ in $\R^N$.
	This concludes the proof. \qed \medskip

\noindent \textbf{Proof of Thereom \ref{spectrum}.} From the definition of $\CE(\BBR^N)$ and the definition of $\gl(\CK_V,\BBR^N)$ in \eqref{lambda}, we deduce that $\CE(\BBR^N) \sbs (-\infty,\gl(\CK_V,\BBR^N) ]$. 

We next prove the reverse inclusion $(-\infty,\gl(\CK_V,\BBR^N) ] \sbs \CE(\BBR^N)$. By Theorem \ref{main0} 2.(ii), $\gl(\CK_V,\BBR^N) \in \CE(\BBR^N)$. It remains to show that $\gl \in \CE(\BBR^N)$ for every $\gl < \gl(\CK_V,\BBR^N)$. Indeed, take $\gl < \gl(\CK_V,\BBR^N)$. For any $n \in \BBN$, let $f_n \in C^\infty(B_n)$ be  nonnegative and not identically equal to zero in $B_n$ with $\supp f_n \in B_n \sms B_{n-1}$. Since $\gl(\CK_V,B_n)>\gl(\CK_V,\BBR^N)>\gl$, it follows that $\gl(\CK_{V-\gl},B_n)>0$. By \cite[Theorem 2 (v)]{GS}, there exists a unique nonnegative weak solution $u_n \in W_0^{1,p}(B_n)$ of
$$ \left\{  \BAL \CK_{V-\gl}[u_n] &= f_n \quad &&\text{in } B_n \\
u &= 0 &&\text{on } \prt B_n.
\EAL \right.
$$
By the strong maximum principle \cite[Theorem 2 (ii)]{GS}, we obtain that $u_n$ is positive in $B_n$. Put 
$$ v_n(x):=\frac{u_n(x)}{u_n(0)} $$
then $v_n(0)=1$. By  the Harnack inequality \cite{Tr} and regularity results for quasilinear elliptic equations \cite{Di}, up to a subsequence, the sequence $\{v_n\}$ converges in $C_{loc}^1(\BBR^N)$ to a function $v$ which is a nonnegative weak solution of $\CK_{V-\gl}[v]=0$ in $\BBR^N$. Since $v(0)=1$, by the Harnack inequality, we deduce that $v$ is positive in $\BBR^N$. Therefore $v$ is a positive weak solution of \eqref{eigenRN} in $\BBR^N$. It follows that $\gl \in \CE(\BBR^N)$. Finally $(-\infty,\gl(\CK_V,\BBR^N) ] \sbs \CE(\BBR^N)$. This completes the proof. \qed \medskip

\noindent \textbf{Proof of Theorem \ref{MP}.} Suppose by contradiction that $u$ is positive somewhere in $\R^N$.
Since $\gl''(\CK_V,\R^N) >0$, there exist a function $\phi\in C^{1}_{loc}(\R^N)$ and positive number $\lambda$ and $\beta$  such that $\inf_{\R^N}\phi\geq \beta>0$ and 
$$\CK_V [\phi]\geq \lambda\phi^{p-1}\quad\quad\text{in the weak sense in } \R^N.$$

 Since $u$ is continuous, $\sup_{\R^N}u<\infty$ and $\limsup_{|x|\to\infty}u(x)\leq0$, we can find a positive constant $\gamma>0$ and $x_0\in\R^N$ such that
$$0<\gamma=\frac{u(x_0)}{\phi(x_0)}=\max_{\R^N}\frac{u(x)}{\phi(x)}<\infty.$$
It follows that there exists $r>0$ such that $u(x)>0$ in $\ol B_r(x_0)$, $u(x)\leq \gamma\phi(x)$, $\forall x\in \ol B_r(x_0)$ and $u(x_0)=\gamma\phi(x_0)$. Therefore, by strong comparison principle \cite[Theorem 3.2]{FrPi}, we get $\gamma\phi=u$ in $B_r(x_0)$. Therefore, for every $0<\psi \in C_c^1(B_r(x_0))$,
\bel{contra1} \int_{B_r(x_0)}\CK_V[\gg \gf]\psi \,dx= \int_{B_r(x_0)}\CK_V[u]\psi \,dx.
\ee
On the other hand, we have
\begin{equation} \label{mp1}
\CK_V[\gamma\phi]\geq \lambda\gamma^{p-1}\beta^{p-1}>0\geq \CK_V[u] \quad\text{ in the weak sense in }\in \BBR^N.
\end{equation}
This implies
\bel{contra2} \int_{B_r(x_0)}\CK_V[\gg \gf]\psi \,dx> \int_{B_r(x_0)}\CK_V[u]\psi \,dx,
\ee
which contradicts with \eqref{contra1}. Thus $u$ must be nonpositive. 

%Conversely, if any $u$ satisfying (\ref{MP0}) implies that $u \leq 0$ in $\BBR^N$. We assume by contradiction that $\gl'(\CK_V,\R^N) <0$. Then there exists $\lambda<0$ and a positive bounded test function $\phi\in W^{1,p}(\R^N)$  satisfying
%$$\CK_V[\phi]\leq \lambda\phi^{p-1}.$$
%Since $p\geq N$, then $\phi$ is continuous in $\R^N$ due to the

\qed \medskip

\bigskip

\noindent{\bf Acknowledgements.}  P-T. Nguyen is supported by Fondecyt Grant 3160207.


\begin{thebibliography}{99}
%	\bibitem{AH}
%	W. Allegretto and Y.X Huang,
%	{\em A Picone's identity for the p-Laplacian and applications}. Nonlinear Anal. {\bf 32} (1998), 819-830. 
	
	
	
%	\bibitem{BDNZ}
%	H. Berestycki, O. Diekman, K. Nagelkerke and P. Zegeling, {\em Can a Species Keep Pace with a Shifting Climate ?}, Bull. Mathematical Biology {\bf 71}  (2009), 399-429.
%	
%	\bibitem{BC}
%	H. Berestycki and G. Chapuisat,
%	{\em  Traveling fronts guided by the environment for reaction-diffusion equations}, 
%	Netw. Heterog. Media {\bf 8} (2013), no. 1, 79-114. 
	
	%\bibitem{BR3}
	%H. Berestycki and L. Rossi, {\em Generalizations and properties of the %principal eigenvalue of elliptic operators in unbounded domains}, %{arXiv:1008.4871}.
	
	%\bibitem{BCX}
	%H. Brezis, M. Chipot and Y. Xie, {\em Some remarks on Liouville type %theorems}, Recent advances in nonlinear analysis 43--65, World Sci. Publ., %Hackensack, NJ, 2008. 
	
	
	\bibitem{BDPR}
	H. Berestycki, I. Capuzzo Dolcetta, A. Porretta, L. Rossi,\newblock {\em Maximum Principle and generalized principal
eigenvalue for degenerate elliptic operators}, J. Math. Pures Appl. {\bf 103} (2015), 1276-1293.
		
%\bibitem{BCV}
%H. Berestycki, J. Coville, H-H. Vo, {\em Persistence criteria for populations with non-local dispersion}, J. Math. Biol. (to appear).  (DOI:
%    10.1007/s00285-015-0911-2).		
%	
%	\bibitem{BHR}
%	H. Berestycki, F. Hamel and L. Roques,
%	\newblock {\em Analysis of the periodically fragmented environment model. {I}.
%		{S}pecies persistence},
%	\newblock {J. Math. Biol.} {\bf 51} (2005), 75-113.
	
%	\bibitem{BR5}
%	H. Berestycki, F. Hamel and L. Rossi, {\em Liouville-type results for semilinear elliptic equations in unbounded domains}, Ann. Mat. Pura Appl. {\bf 186} (2007), 469-507. 
	
	\bibitem{BNV}
	H. Berestycki, L. Nirenberg and S. R. S. Varadhan,
	{\em The principal eigenvalue and maximum principle for second-order
		elliptic operators in general domains}, 
	Comm. Pure Appl. Math. {\bf 47} (1994), 47-92. 
	
	\bibitem{BR0}
	H. Berestycki and L. Rossi,
	{\em On the principal eigenvalue of elliptic operators in $\R^N$ and applications},  
	J. Eur. Math. Soc. (JEMS) {\bf 8} (2006), 195-215. 
	

	
	\bibitem{BR3}
	H. Berestycki and L. Rossi, {\em Generalizations and properties of the principal eigenvalue of elliptic operators in unbounded domains}, Comm. Pure Appl. Math {\bf 68} (2015), 1014-1065.
	%
	\bibitem{BBV} M. F. Bidaut-V\'eron, R. Borghol and L. V\'eron, {\em Boundary Harnack inequality and a priori estimates of singular solutions of quasilinear elliptic equations}, Calc. Var. Partial Differential Equations 27 (2006), no. 2, 159--177.

\bibitem{CLMC}	
	 F. Catt\'e, P-L Lions, J-M.  Morel and T. Coll,  {\em  Image selective smoothing and edge detection by nonlinear diffusion}, SIAM J. Numer. Anal. {\bf 29} (1992), 182-193.
	
%	\bibitem{Da}
%	L. Damascelli, {\em Comparison theorems for some quasilinear degenerate elliptic operators and applications to symmetry and monotonicity results},  Ann. Inst. H. Poincar\'e Anal. Non Lin\'eaire {\bf 15} (1998), 493-516.
	
%	\bibitem{BD0}
%	I. Birindelli and F. Demengel, 
%	\newblock {\em Existence of solutions for semi-linear equations involving the p-Laplacian: the non coercive case}, 
%	\newblock {Calc. Var. Partial Differential Equations} {\bf 20} (2004), 343-366. 
%	
%	\bibitem{BD1}
%	I. Birindelli and F. Demengel, 
%	\newblock {\em Regularity and uniqueness of the first eigenfunction for singular fully nonlinear operators}, 
%	\newblock {J. Differential Equations} {\bf 249} (2010), 1089-1110. 
	
	%\bibitem{BD2}
	%I. Birindelli and F. Demengel, 
	%\newblock{\em Overdetermined problems for some fully non linear operators},
	%Comm. Partial Differential Equations {\bf 38} (2013), no. 4, 608-628. 
	%
	%\bibitem{Bre}
	%H. Brezis, {\em Functional analysis, Sobolev spaces and Partial Differential %Equations}, Springer (2010), xiii+599p.
	%
	
	\bibitem{CDG}
	A. Ca\~nada, P. Dr\'abek and J. L. G\'amez,	{\em Existence of positive solutions for some problems with nonlinear diffusion}, Trans. Amer. Math. Soc. {\bf 349} (1997), 4231-4249.
	
	%\bibitem{Cha}
	%G. Chapuisat
	%\newblock {\em Existence and nonexistence of curved front solution of a %biological equation}.
	%J. Differential Equations 236 (2007), no. 1, 237--279.  
	
%	\bibitem{Da}
%	L. Damascelli, {\em Comparison theorems for some quasilinear degenerate elliptic operators and applications to symmetry and monotonicity results},  Ann. Inst. H. Poincar\'e Anal. Non Lin\'eaire {\bf 15} (1998), 493-516.
%	

	
\bibitem{DFP}	
	B. Devyver, M. Fraas, Y. Pinchover,
	{\em Optimal Hardy weight for second-order elliptic operator: an answer to a problem of Agmon}, J. Funct.
Anal. {\bf 266} (2014) 4422--4489.

\bibitem{DePi}
 B. Devyver and Y. Pinchover, {\em Optimal $L^p$ Hardy-type inequalities}, Ann. Inst. H. Poincar´e.
Anal. Non Lineaire {\bf 33} (2016), 93--118.
	%
	
	%
%	\bibitem{DFSV}
%	L. Damascelli, A. Farina, B. Sciunzi  and E. Valdinoci, {\em Liouville results for m-Laplace equations of Lane-Emden-Fowler type}, Ann. Inst. H. Poincar\'e Anal. Non Lin\'eaire {\bf 26} (2009), 1099-1119. 
	%
	\bibitem{Di} E. Dibenedetto, {\em $C^{1+\alpha}$ local regularity of weak solutions of degenerate elliptic equations}, Nonlinear Anal. {\bf 7}  (1983), 827-850. 
%	%
\bibitem{DKN} P. Dr\'abek, A. Kufner and F. Nicolosi, {\em Quasilinear elliptic equations with degenerations and singularities}, de Gruyters Series in Nonlinear Analysis and Applications {\bf 5}, Walter de Guyter $\&$ Co., Berlin, 1997.
%
\bibitem{DrRa} Pavel Dr\'abek and S. H. Rasouli, {\em A Quasilinear Eigenvalue Problem with Robin Conditions on the Non-Smooth Domain of Finite Measure}, Zeitschrift f\" ur Analysis und ihre Anwendungen {\bf 29}, 469-485.
%
\bibitem{DuGu} Y. Du and Z. M. Guo, {\em Boundary blow-up solutions and their applications in quasilinear elliptic equations}, J. Anal. Math. {\bf 89}  (2003), 277-302. 
%
	\bibitem{FrPi} M. Fraas and Y. Pinchover, {\em Positive Liouville Theorem and Asymptotic behavior for $p$-Laplacian type elliptic equations with a Fuchsian potential}, Confluentes Mathematici {\bf 3} (2011) 291-323.

\bibitem{FO}
Y. Furusho, Y. Ogura
 {\em On the existence of bounded positive solutions of semilinear elliptic equations in exterior domains.}
Duke Math. J. {\bf 48} (1981), 497-521. 	
	
	%\bibitem{GSq}
	%L. Jeanjean, M. Squassina
	%{\em Existence and symmetry of least energy solutions for a class of %quasi-linear elliptic equations}.
	%Ann. Inst. H. Poincar\'e Anal. Non Lin\'eaire 26 (2009), no. 5, 1701--1716. 
	
	\bibitem{GS} J. Garc\'ia-Meli\'an and J. Sabina de Lis, {\em Maximum and Comparison principles for operators involving the $p-$Laplacian}, J. Math. Anal. Appl. {\bf 218} (1998), 49-65.
	%
	\bibitem{GTbook} D. Gilbarg and N. Trudinger, Elliptic partial differential equations of second order, 2nd edn (Springer, Berlin, 1983).
	%
%\bibitem{K1}	
%R. Kajikiya
%{\em A priori estimate for the first eigenvalue of the p-Laplacian}.	Differential Integral Equations 28 (2015), no. 9-10, 1011--1028. 

	
	\bibitem{KLP} B. Kawohl, M. Lucia and S. Prashanth, {\em Simplicity of the principal eigenvalue for indefinite
		quasilinear problems}, Adv. Differential Equations {\bf 12} (2007), 407-434.
	%
	\bibitem{La} O.A. Ladyzhenskaya and N. N.  Ural\'tseva,  {\em Linear and quasilinear elliptic equations}. Translated from the Russian by Scripta Technica, Inc. Translation editor: Leon Ehrenpreis Academic Press, New York-London, 1968.
	%
\bibitem{LS} V. K. Le and K. Schmitt, {\em Sub-Supersolution theorems for quasilinear elliptic problems: A variational approach}, Elec. J. Diff. Equ. {\bf 118} (2004), 1-7.
	%
	
	%\bibitem{LWZ}
	%Y. Li, ZQ.  Wang, J. Zeng,
	%\newblock {\em Ground states of nonlinear Schrödinger equations with %potentials}. 
	%Ann. Inst. H. Poincar\'e Anal. Non Lin\'eaire 23 (2006), no. 6, 829--837.
	
	%\bibitem{Mal}
	%A. Malchiodi.
	%\newblock {\em Some new entire solutions of semilinear elliptic equations on %$\R^N$}. 
	%Adv. Math. 221 (2009), no. 6, 1843--1909. 
\bibitem{Lieberman}	
GM.	Lieberman,  
\newblock {\em Boundary regularity for solutions of degenerate elliptic equations.} Nonlinear Anal. {\bf 12} (1988), 1203-1219. 
	
	\bibitem{Lin} P. Lindqvist, 
	\newblock {\em On the equation $\emph{div}(\nabla u|^{p-2}\nabla u)+λ|u|^{p-2}u=0$},
	\newblock Proc. Amer. Math. Soc. {\bf 109} (1990), 157-164.
	
\bibitem{Mao}	
J. Mao,
\newblock {\em Eigenvalue inequalities for the p-Laplacian on a Riemannian manifold and estimates for the heat kernel.} 
J. Math. Pures Appl. {\bf 101} (2014), no. 3, 372--393. 	
\bibitem{Murata}	
M. Murata, {\em  Structure of positive solutions to $(-\Delta+V)u=0$ in $\R^N$.} Duke Math. J. {\bf 53} (1986), 869-943. 

	
	\bibitem{NV} P-T. Nguyen and H-H. Vo,  
	\newblock {\em Existence, uniqueness and qualitative properties of positive solutions of quasilinear elliptic equations},
	\newblock J. Funct. Anal. {\bf 269} (2015), 3120-3146. . 
	%
	\bibitem{PT} Y. Pinchover and K. Tintarev, {\em Ground state alternative for $p$-Laplacian with potential term}, Calc. Var. Partial Differential Equations {\bf 28} (2007), 179-201.
	%
\bibitem{Pinchover}
    Y. Pinchover
  {\em    On positivity, criticality, and the spectral radius of the shuttle operator for elliptic operators} Duke Math. J. {\bf 85} (1996), 431-445	

\bibitem{Pinchover1} Y. Pinchover, {\em   Liouville-type theorem for Schrödinger operators}. Comm. Math. Phys. {\bf 272} (2007), 75-84
%
\bibitem{PuSe1} P. Pucci and J. Serrin, {\em A note on the strong maximum principle for elliptic differential inequalities}, J. Math. Pures Appl. {\bf 79} (2000), 57-71.	
%
\bibitem{PuSe2} P. Pucci and J. Serrin, {\em The strong maximum principle revisited}, J. Differential Equations {\bf 196} (2004), 1-66.
%
\bibitem{PSZ} P. Pucci, J. Serrin and H. Zou, {\em A strong maximum principle and a compact support principle for singular elliptic inequalities}, J. Math. Pures Appl. {\bf 78} (1999), 769-789.
	
	
	
	
	\bibitem{QS}
	A. Quaas and B. Sirakov,
	\newblock {\em Principal eigenvalues and the Dirichlet problem for fully nonlinear elliptic operators}, 
	Adv. Math. {\bf 218} (2008), 105-135. 
	
	\bibitem{Se} J. Serrin, {\em Local behavior of solutions of quasi-linear equations}, Acta Math. {\bf 111} (1964), 247-301.
	%
	\bibitem{SW}
R. E..Showalter, Walkington, N. J. 
{\em  Diffusion of fluid in a fissured medium with microstructure}.
SIAM J. Math. Anal. {\bf 22} (1991), 1702-1722. 


%\bibitem{JLM}P. Juutinen ,P. Lindqvist and J J. Manfredi, {\em The $\infty$-eigenvalue problem}. Arch. Ration. Mech. Anal. {\bf 148} (1999), no. 2, 89-105. 	
	
	\bibitem{Tol} P. Tolksdorf, {\em Regularity for a more general class of quasilinear elliptic equations}, J. Diff. Eq. {\bf 51}, 126-150 (1984).
	%
	\bibitem{Tr} N. S. Trudinger, {\em On Harnack type inequalities and their application to quasilinear elliptic equations}, Comm. Pure Appl. Math. {\bf XX} (1967), 721-747.
	%
%\bibitem{V1}
%H-H.Vo, {\em Persistence versus extinction under a climate change in mixed environments}, J. Differential Equations {\bf 259} (2015), 4947-4988.		
	
%	\bibitem{Zh} G. Zhang, {\em Weighted Sobolev spaces and ground state solutions for quasilinear elliptic problems with unbounded and decaying potentials}, Bound. Value Probl. (2013), 2013:189.
	

	
\end{thebibliography}
\end{document}